\documentclass{amsart}
\usepackage[utf8]{inputenc}
\usepackage{todonotes} 
\usepackage{amsmath, amssymb, amsthm} 
\usepackage{latexsym} 
\usepackage{dsfont}
\usepackage[flushleft]{paralist}
\usepackage[all]{xy}
\usepackage{tikz-cd}
\usepackage{stmaryrd} 
\usepackage{hyperref} 
\usepackage[OT2,T1]{fontenc}
\tikzcdset{scale cd/.style={every label/.append style={scale=#1},
    cells={nodes={scale=#1}}}}
\DeclareSymbolFont{cyrletters}{OT2}{wncyr}{m}{n}
\DeclareMathSymbol{\Zhe}{\mathalpha}{cyrletters}{"11} 
    
\begin{document}
\title{On the growth of the Jacobians in $\Z_p^l$-voltage covers of graphs}
\author[S.~Kleine]{Sören Kleine} 
\address[Kleine]{Institut für Theoretische Informatik, Mathematik und Operations Research, Universität der Bundeswehr München, Werner-Heisenberg-Weg 39, 85577 Neubiberg, Germany} 
\email{soeren.kleine@unibw.de} 
\author[K.~Müller]{Katharina Müller}
\address[Müller]{D\'epartement de Math\'ematiques et de Statistique\\
Universit\'e Laval, Pavillion Alexandre-Vachon\\
1045 Avenue de la M\'edecine\\
Qu\'ebec, QC\\
Canada G1V 0A6}
\email{katharina.mueller.1@ulaval.ca}


\newcommand{\R}{\mathds{R}}
 	\newcommand{\Z}{\mathds{Z}}
 	\newcommand{\N}{\mathds{N}}
 	\newcommand{\Q}{\mathds{Q}}
 	\newcommand{\K}{\mathds{K}}
 	\newcommand{\M}{\mathds{M}} 
 	\newcommand{\C}{\mathds{C}}
 	\newcommand{\B}{\mathds{B}}
 	\newcommand{\LL}{\mathds{L}}
 	\newcommand{\F}{\mathds{F}}
 	\newcommand{\p}{\mathfrak{p}} 
 	\newcommand{\q}{\mathfrak{q}} 
 	\newcommand{\f}{\mathfrak{f}} 
 	\newcommand{\Pot}{\mathcal{P}}
 	\newcommand{\Gal}{\textup{Gal}}
 	\newcommand{\rg}{\textup{rank}}
 	\newcommand{\id}{\textup{id}}
 	\newcommand{\Ker}{\textup{Ker}}
 	\newcommand{\Image}{\textup{Im}} 
 	\newcommand{\pr}{\textup{pr}}
 	\newcommand{\la}{\langle}
 	\newcommand{\ra}{\rangle}
 	\newcommand{\gdw}{\Leftrightarrow}
 	\newcommand{\pfrac}[2]{\genfrac{(}{)}{}{}{#1}{#2}}
 	\newcommand{\Ok}{\mathcal{O}}
 	\newcommand{\Norm}{\mathrm{N}} 
 	\newcommand{\coker}{\mathrm{coker}}
 	\newcommand{\dotcup}{\stackrel{\textstyle .}{\bigcup}}
 	\newcommand{\Cl}{\mathrm{Cl}}
 	\newcommand{\Sel}{\textup{Sel}} 
 	\newcommand{\OkG}{\Ok\llbracket G\rrbracket }
 	\newcommand{\Ann}{\textup{Ann}} 

 	\newtheorem{lemma}{Lemma}[section] 
 	\newtheorem{prop}[lemma]{Proposition} 
 	\newtheorem{defprop}[lemma]{Definition and Proposition} 
 	\newtheorem{conjecture}[lemma]{Conjecture} 
 	\newtheorem{thm}[lemma]{Theorem} 
 	\newtheorem*{thm*}{Theorem} 
 	\newtheorem{cor}[lemma]{Corollary}
 	\newtheorem{claim}[lemma]{Claim}

 	\theoremstyle{definition}
 	\newtheorem{def1}[lemma]{Definition} 
 	\newtheorem{ass}[lemma]{Assumption}
 	\newtheorem{rem}[lemma]{Remark} 
 	\newtheorem{rems}[lemma]{Remarks} 
 	\newtheorem{example}[lemma]{Example} 
 	\newtheorem{fact}[lemma]{Fact}

\maketitle

\begin{abstract}  
  We investigate the growth of the $p$-part of the Jacobians in voltage covers of finite connected multigraphs, where the voltage group is isomorphic to $\Z_p^l$ for some ${l \ge 2}$, and we study analogues of a conjecture of Greenberg on the growth of class numbers in multiple $\Z_p$-extensions of number fields. Moreover we prove an Iwasawa main conjecture in this setting, and we study the variation of (generalised) Iwasawa invariants as one runs over the $\Z_p^l$-covers of a fixed finite graph $X$. We discuss many examples; in particular, we construct examples with non-trivial Iwasawa invariants. 
\end{abstract}

\section{Introduction} \label{section:intro} 
Let $K_\infty$ be a $\Z_p$-extension of a number field $K$. Then for each ${m \in \N}$ there exists a unique subextension $K_m/K$ of degree $p^m$. Let $h_m$ be the class number of $K_m$. Then Iwasawa (see \cite{iwasawa}) proved that we have an asymptotic formula 
\begin{align} \label{eq:iw} 
  v_p(h_m) = \mu(K_\infty/K) \cdot p^m + \lambda(K_\infty/K) \cdot m + \nu(K_\infty/K) 
\end{align} 
for each sufficiently large $m$, where $\mu(K_\infty/K)$, $\lambda(K_\infty/K)$ and $\nu(K_\infty/K)$ denote the Iwasawa invariants of the $\Z_p$-extension $K_\infty/K$. More generally, let $K_\infty$ be a $\Z_p^l$-extension of a number field $K$, with intermediate fields $K_m$. Then Greenberg conjectured that there exists a polynomial ${P(X,Y) \in \Q[X,Y]}$ of total degree at most $l$ and of degree at most 1 in $Y$ such that 
\[ v_p(h_m) = P(p^m,m) \] 
for each sufficiently large $m$, where $h_m$ denotes the class number of $K_m$, ${m \in \N}$. In this paper, we will refer to this conjecture as \emph{Greenberg's conjecture}. 

In Iwasawa theory of graphs one starts with a finite connected graph $X$ instead of a number field $K$ and considers a sequence of Galois covers
\[X\leftarrow X_1\leftarrow X_2\leftarrow X_3\leftarrow X_4\leftarrow \dots \] 
such that $\Gal(X_m/X)\cong (\Z/p^m\Z)^l$. Then $X_m$ is again a finite graph and we will always assume that each $X_m$ is connected. Instead of the $p$-part of the class group of the field $K_m$ one considers the $q$-part of the number of spanning trees of $X_m$ (where $q$ is a prime not necessarily distinct from $p$). Vallières and McGown-Vallières proved in a sequence of papers (see \cite{vallieres1}, \cite{vallieres2}, \cite{vallieres3}) that (under mild technical hypothesis) the number of spanning trees grows asymptotically as in Greenberg's conjecture in the case $l=1$ and ${p=q}$. In a subsequent work Lei and Vallières considered the case $l=1$ and ${q \neq p}$ (see \cite{vallieres-lei}).

In this paper we study analogous growth patterns in the theory of voltage covers of graphs for $l\ge2$ and ${q =p}$. Our approach is purely algebraic and uses the fact that the number of spanning trees is given by the size of the Jacobian of a finite connected graph (see Section~\ref{section:notation} for precise definitions). For the case $l=1$ this strategy was used to prove an Iwasawa formula like \eqref{eq:iw} for the $p$-part of the number of spanning trees by Gonet \cite{diss-gonet}. We are able to prove Greenberg's conjecture under the assumption that our base graph $X$ is connected (see Theorem \ref{thm:greenberg-conj}). A result of this form had been obtained before by Vallières and DuBose, generalising the approach of McGown-Vallières, using Ihara L-functions \cite{dubose-vallieres}. 

A central aspect of classical Iwasawa theory is the relation between algebraic objects such as certain Galois groups and analytic objects such as $L$-functions (e.g. the Hasse-Weil $L$-function $L(E,s)$ for an elliptic curve $E$ or the Dedekind zeta-function $\zeta_K(s)$ for a number field $K$). In many cases one can define a $p$-adic function interpolating special values of these complex analytic functions. Given a finite connected graph $X$ and a finite Galois covering $Y/X$ we consider the complex analytic Ihara L-function $L_X(\chi,s)$ for every character ${\chi\in \widehat{\Gal(Y/X)}}$. 

 Let now $X$ be a finite connected multigraph, and suppose that $X$ has $n$ vertices. Let $X_m$ be the $m$-th intermediate graph in a $\Z_p^l$-cover $X_\infty$ of $X$. Then there exists an element $\Delta_\infty\in \textup{Mat}_{n,n}(\Z_p[[\Gal(X_\infty/X)]])$, such that $\chi(\det(\Delta_\infty))$ interpolates the "algebraic part" of $L_X(\chi,s)$ at $s=1$. In this setting we obtain the following 
 \begin{thm*}[Iwasawa main conjecture (Theorem~\ref{thm:main_conjecture})]
 Let $J(X_\infty)$ be the Jacobian of $X_\infty$ and let $\Lambda=\Z_p[[\Gal(X_\infty/X)]]$. Then 
 \[\textup{Char}_\Lambda(J(X_\infty)\otimes \Lambda)=\det(\Delta_\infty).\]
 \end{thm*}
Here $\textup{Char}(M)$ denotes the characteristic ideal of a finitely generated and torsion $\Lambda$-module $M$ (for details we refer to Section~\ref{section:notation_iw-modules}). 

The third central topic we want to address in the present paper is the behaviour of so-called \emph{(generalised) Iwasawa invariants} if we vary the $\Z_p^l$-cover of a given finite connected graph $X$. Let $P(X,Y)$ be the polynomial appearing in Greenberg's conjecture. Then
\[P(p^m,m)=m_0p^{ml}+l_0mp^{m(l-1)}+O(p^{m(l-1)}).\]
Following Cuoco and Monsky (see \cite{cuoco-monsky}) we call $m_0$ and $l_0$ the (generalised) Iwasawa invariants of $J(X_\infty)\otimes \Lambda$. In Section~\ref{sec:Fukuda} we define a topology on the set $\mathcal{E}^l(X)$ of voltage $\Z_p^l$-covers of $X$ and we prove the following results. 
\begin{thm*}[Theorems~\ref{thm:local_max1} and \ref{thm:local_max2}]
We have the following two results. 
\begin{itemize}
    \item[(1)]Assume that $l=1$ and let ${X_\infty \in \mathcal{E}^1(X)}$. Then there exists a (sufficiently small) neighbourhood $U$ of $X_\infty$ such that the following statements hold: \begin{compactenum}[(a)] 
      \item For each ${\tilde{X}_\infty \in U}$, we have 
      \[ m_0(\tilde{X}_\infty) \le m_0(X_\infty). \] 
      \item ${l_0(\tilde{X}_\infty) = l_0(X_\infty)}$ for each ${\tilde{X}_\infty \in U}$ for which ${m_0(\tilde{X}_\infty) = m_0(X_\infty)}$. 
    \end{compactenum} 
    \item[(2)]Fix an element ${X_\infty \in \mathcal{E}^l(X)}$, ${l \ge 2}$. Then there exist an integer ${k \in \N}$ and a neighbourhood $U$ of $K_\infty$ such that the following two statements hold for each ${\tilde{X}_\infty}$ in $U$. \begin{compactenum}[(a)] 
     \item $m_0(\tilde{X}_\infty) \le m_0(X_\infty)$, and 
     \item $l_0(\tilde{X}_\infty) \le k$ holds if ${m_0(\tilde{X}_\infty) = m_0(X_\infty)}$. 
   \end{compactenum} 
\end{itemize}
\end{thm*}

Let us briefly describe the structure of the article. Section~\ref{section:notation} is preliminary in nature -- here we introduce the basic notation and the setup. In Section~\ref{section:J(X_m)} we relate the Jacobians of the Galois covers $X_m/X$, ${m \in \N}$, to certain quotients of an Iwasawa module. This enables us to describe the growth of these Jacobians by using results of Cuoco and Monsky, and we already obtain a weak result in the direction of Greenberg's conjecture (see Lemma~\ref{lemma:NundL}). In Section~\ref{section:linearalgebra}, we prove Greenberg's conjecture by studying the voltage $p$-Laplacian of the Galois cover and using a suitable matrix representation in order to describe the growth of the Jacobians in terms of a certain power series. In Section~\ref{section:main_conjecture} we prove the Iwasawa main conjecture and in Section~\ref{sec:Fukuda} we prove Theorems \ref{thm:local_max1} and \ref{thm:local_max2} which have been stated above. The proofs are completely algebraic, and they are based on work of Fukuda (see \cite{fukuda}) and the first named author (see \cite{local_beh} and \cite{local_max}). In the final two sections, we compute many numerical examples by applying the results from Sections~\ref{section:linearalgebra} and ~\ref{section:main_conjecture}. In particular, we construct examples of Galois covers $X_\infty/X$ of graphs with non-trivial $m_0$- and $l_0$-invariant. 

In the classical Iwasawa theory of $\Z_p$-extensions of number fields, Iwasawa himself was the first who constructed examples (in the case of ${l = 1}$) with a non-trivial $m_0$-invariant (see \cite{iwasawa_mu}). On the other hand, to the authors' knowledge no example of a $\Z_p^l$-extension of a number field $K$ is known where one can show that the $l_0$-invariant is greater than zero (however, there exist such examples in the setting of Selmer groups of elliptic curves, see \cite{kleine-matar}). Therefore the construction of an example with non-trivial $l_0$-invariant in the final section might be of some interest. 

\textbf{Acknowledgements.}
Part of this research was carried out during a visit of the second named author at Universit\"at der Bundeswehr M\"unchen and Ludwig-Maximilians-Universit\"at M\"unchen in August and September 2022. The authors would like to thank both institutions for their hospitality and in particular Cornelius Greither and Werner Bley for fruitful discussions concerning the content of this paper. 
The authors would also like to thank Anwesh Ray, Cédric Dion and in particular Daniel Vallières for their comments on this work. We are grateful to Sage DuBose and Daniel Vallières for sharing with us a preprint of their work. 
The second named authors's research is supported by the NSERC Discovery Grants Program RGPIN-2020-04259 and RGPAS-2020-00096.

\section{Notation and definitions} \label{section:notation} 
\subsection{Iwasawa modules} \label{section:notation_iw-modules} 
Let $\Gamma$ be a group which is topologically isomorphic to $\Z_p^l$. Then the completed group ring $\Z_p\llbracket \Gamma\rrbracket $ can be identified (non-canonically) with the ring $\Z_p\llbracket T_1, \ldots, T_l\rrbracket $ of formal power series in $l$ variables, as follows: choose a set of topological generators ${\gamma_1, \ldots, \gamma_l}$ of $\Gamma$, and consider the bijective homomorphism between the group rings which is induced by mapping $\gamma_i$ to ${T_i + 1}$, respectively. 

In the following we always denote by $\Lambda_l = \Z_p\llbracket T_1, \ldots, T_l \rrbracket $ the Iwasawa algebra in $l$ variables. If $l$ is clear from the context, then we abbreviate $\Lambda_l$ to $\Lambda$. Sometimes we will also work with the ring ${R = \Z_p[T_1, \ldots, T_l]}$. 

Now we describe the structure theory of \emph{Iwasawa modules}. In this article, any finitely generated $\Lambda$-module $A$ will be called an Iwasawa module. An Iwasawa module is called \emph{pseudo-null} if it is annihilated by two relatively prime elements of the unique factorisation domain $\Lambda$. A \emph{pseudo-isomorphism} of finitely generated $\Lambda$-modules $A$ and $B$ is a $\Lambda$-module homomorphism ${\varphi \colon A \longrightarrow B}$ such that the kernel and the cokernel of $\varphi$ are pseudo-null. 

Let now $A$ be any Iwasawa module. Then by the general structure theory (see \cite[Proposition~5.1.7(ii)]{nsw}) there exists an \emph{elementary} $\Lambda$-module 
\[ E_A = \bigoplus_{i = 1}^s \Lambda/(p^{e_i}) \oplus \bigoplus_{j = 1}^t \Lambda/(g_j) \] 
and a pseudo-isomorphism ${\varphi \colon A \longrightarrow E_A}$. Here $g_1, \ldots, g_t$ are elements of $\Lambda$ which are coprime with $p$. The \emph{characteristic power series} $F_A$ attached to $A$ is the element 
\[ F_A = p^{\sum_{i = 1}^s e_i} \cdot \prod_{j = 1}^t g_j, \] 
and the \emph{characteristic ideal} $\textup{Char}(E_A)$ of $E_A$  (and $A$ respectively) is the principal ideal $(F_A)$. The characteristic power series is determined up to units of $\Lambda$ by the Iwasawa module $A$, and therefore the characteristic ideal is well-defined. 

Moreover, we can define the \emph{(generalised) Iwasawa invariants} of $A$ as follows. To this purpose, we fix an isomorphism between $\Z_p\llbracket \Gamma\rrbracket $ and $\Lambda$, as above, and we consider the characteristic power series $F_A$ attached to $A$. Let 
\[ m_0(A) = \sum_{i = 1}^s e_i, \] 
i.e. $p^{m_0(A)}$ is the exact power of $p$ which divides $F_A$. Moreover, we write 
\[ F_A = p^{m_0(A)} \cdot G_A, \] 
i.e. $G_A$ is coprime with $p$, and we consider the coset $\overline{G_A}$ of $G_A$ in the quotient algebra $\overline{\Lambda} = \Lambda/(p)$. Then we define 
\[ l_0(A) = \sum_{\mathcal{P}} v_{\mathcal{P}}(\overline{G_A}), \] 
where the sum runs over all the prime ideals of $\overline{\Lambda}$ of the form ${\mathcal{P} = (\overline{\sigma - 1})}$ for some element ${\sigma \in \Gamma \setminus \Gamma^p}$, and where $v_{\mathcal{P}}$ denotes the $\mathcal{P}$-adic valuation. Note that this sum is finite since $\overline{\Lambda}$ is again a unique factorisation domain. 

In the case of ${l = 1}$, the structure theory of Iwasawa modules is better understood. In this case, an Iwasawa module is pseudo-null if and only if it is \emph{finite}. Moreover, it follows from the Weierstraß Preparation Theorem for ${\Lambda = \Z_p\llbracket T\rrbracket }$ (see \cite[Theorem~7.3]{wash}) that the characteristic power series $F_A$ is associated to a \emph{polynomial} of the form 
\[ p^{m_0(A)} \cdot G_A, \] 
where $G_A$ is \emph{distinguished}. This means that $G_A$ is monic, and that each coefficient of $G_A$, apart from the leading one, is divisible by $p$. Then we define the \emph{classical Iwasawa invariants} of $A$ as 
\[ \mu(A) = m_0(A), \quad \lambda(A) = \deg(G_A)=l_0(G_A). \] 
If $A = \varprojlim A_n$ is the classical Iwasawa module of ideal class groups in the intermediate layers of a $\Z_p$-extension $K_\infty/K$, then the Iwasawa invariants of $A$ are the constants from the asymptotic growth formula~\eqref{eq:iw} mentioned in the introduction (see \cite{iwasawa} or \cite[Chapter~13]{wash} for more details). 

We conclude the current section by introducing some more notation. For any ${i \in \{1, \ldots, l\}}$ and every ${m \in \N}$, we let 
\[ \omega_m(T_i) = (T_i + 1)^{p^m} - 1. \] 
We also denote by $I_m$ the ideal of ${R \subseteq \Lambda}$ generated by all the $\omega_m(T_i)$, and we denote the quotient ${R/I_m}$ by $R_m$, ${m \in \N}$. 

\subsection{Graphs} 
In this article a graph $Y$ will always be a connected multigraph.
We denote the edges between the vertices $v$ and $w$ by $e_i(v,w)$. Note that for fixed $v$ and $w$ the set $E(v,w)$ of edges between $v$ and $w$ might contain more than one element. For any graph $Y$ we make the following definitions. 
\begin{def1}
Let $Y$ be a (not necessarily finite) graph. Then we denote by $V(Y)$ the set of vertices of $Y$. We assume that $\deg(v)$ is finite for each ${v \in V(Y)}$ (such graphs are called locally finite), where 
\[ \deg(v) = \sum_{w \in V(Y), w\neq v} |E(v,w)|+2\vert \{\textup{loops from $v$ to $v$}\}\vert\]
means the \emph{degree} of $v$ and $\textup{mult}(v,w) = |E(v,w)|$ means the number of edges between $v$ and $w$ (i.e. $\textup{mult}(v,w) = 0$ if $v$ and $w$ are not adjacent). We define
\begin{align*}
    \textup{Div}(Y) &=\left\{\sum_{v \in V(Y)} a_v v \; \Big| \; a_v\in \Z_p, a_v=0 \textup{ for all but finitely many }v\right\}\\
    \textup{Div}^0(Y) &=\left\{d = \sum_{v \in V(Y)} a_v v \in \textup{Div}(Y) \; \Big| \; \sum_{v \in V(Y)} a_v=0\right\} \\ 
    \textup{Pr}(Y) &=\{d\in \textup{Div}^0(Y)\, \mid \, d\sim 0\}, 
\end{align*}
where $d\sim d'$ if there is a sequence of firing moves that transforms $d$ into $d'$ (see \cite[Definition 1.5]{corry-Perkinson}).
We further define the two quotient groups 
\begin{align*}
    \textup{Pic}_p(Y) &=\textup{Div}(Y)/\textup{Pr}(Y), \\
    J_p(Y) &=\textup{Div}^0(Y)/\textup{Pr}(Y).
\end{align*}
\end{def1}
Note that $J_p(Y)$ is the $p$-Sylow subgroup of the \emph{Jacobian} of $Y$. As the prime $p$ is fixed once and for all we will write $J(Y)$ instead of $J_p(Y)$ from now on. If $Y$ is finite and connected, then it is well-known that $J(Y)$ is a finite abelian group (see \cite[Proposition 2.37]{corry-Perkinson}). Furthermore, the group $\textup{Pr}(Y)$ is generated by the elements
\[p_v=\deg(v)v-\sum_{w\sim v, w\neq v}\textup{mult}(v,w)w-2\vert \{\textup{loops from $v$ to $v$}\}\vert v\quad \forall v\in V(Y).\]

Let now $X$ be a finite graph with $n$ vertices. We specialise to the situation of voltage covers $X_m$ (see also \cite{diss-gonet}). Let $\mathbf{E}_X$ and $E_X$ be the sets of directed and undirected edges, of $X$ and let $\gamma\colon E_X\longrightarrow\mathbf{E}_X$ be a section of the natural map $\mathbf{E}_X\longrightarrow E_X$. Let $S=\gamma(E_X)$. We fix a \emph{voltage assignment}, i.e. a map  
\[\alpha \colon S\longrightarrow \Z_p^l.\]
Using the natural projection maps ${\Z_p^l\longrightarrow (\Z/p^m\Z)^l}$, ${m \in \N}$, $\alpha$ induces well-defined assignments
\[\alpha_m\colon S\longrightarrow (\Z/p^m\Z)^l.\]
Since the target groups of $\alpha$ in our applications usually will be Galois groups, we write the target groups of $\alpha$ and of the $\alpha_m$ multiplicatively (although this might look strangely). Note that for each element $e\in \mathbf{E}_X$ either $e\in S$ or $e^{-1}\in S$. If $e\notin S$, we define $\alpha(e)=\alpha(e^{-1})^{-1}$ and analogously for $\alpha_m$. Thus, we can interpret $\alpha$ and $\alpha_m$ as maps defined on $\mathbf{E}_X$. 
\begin{def1} \label{def:voltagelevel} 
  We define the derived graph ${X_m := X((\Z/p^m\Z)^l,S,\alpha_m)}$ on the vertices $(v,g\!\!\pmod{p^m})$, where $g\in \Z_p^l$ and $g \!\!\pmod{p^m}$ denotes the equivalence class of $g$ in $(\Z/p^m\Z)^l$. We draw an edge between $(v , g\!\!\pmod{p^m})$ and $(v', g'\!\!\pmod{p^m})$ if there is an edge $e$ from $v$ to $v'$ in $\mathbf{E}_X$ such that ${g'\pmod{p^m}=g \cdot \alpha_m(e) \pmod{p^m}}$. 
  
  We abbreviate the vertices $(v, 1)$ of the base graph to $v$.  
  \end{def1} 
  We have a natural action of $(\Z/p^m\Z)^l$ on $X_m$ given by
\[(g'\!\!\!\!\pmod{p^m}) \circ (x,g\!\!\!\!\pmod{p^m})=(x,g \cdot g'\!\!\!\!\pmod{p^m}).\] 
By using the natural projection $\Z_p^l\longrightarrow (\Z/p^m\Z)^l$, we also have a natural action of $\Z_p^l$ on $X_m$. This induces a well-defined action of $\Z_p^l$ on $J(X_m)$. We also obtain an action of the two canonical group rings 
\[ R := \Z_p[\Z_p^l], \quad \Z_p[[\Z_p^l]]. \] 
The ring $\Z_p[\Z_p^l]$ is isomorphic non-canonically to the ring $R = \Z_p[T_1, \ldots, T_l]$ which has been defined in Section~\ref{section:notation_iw-modules}, and the latter completed group ring is isomorphic non-canonically to the Iwasawa algebra ${\Lambda_l = \Z_p[[T_1, \ldots, T_l]]}$, as has been explained in the previous subsection. 
\begin{def1} 
   We let $X_\infty$ be a $\Z_p^l$-voltage cover of $X$, with intermediate graphs $X_m$. The projective limit 
   \[ J_\infty(X_\infty) := J(X_\infty) \otimes_R \Lambda_l \] 
   is a finitely generated $\Lambda_l$-module, i.e. it is an Iwasawa module in the language of Section~\ref{section:notation_iw-modules}. Therefore it makes sense to speak of the (generalised) Iwasawa invariants of the graph $X_\infty$. 
\end{def1} 

In this article, we always assume that all the graphs $X_m$ are connected. We will provide a sufficient criterion for this property below (see Lemma~\ref{lemma:connected}). Thus, $X_m/X$ is a Galois cover with group $(\Z/p^m\Z)^l$. Then as a $\Z_p[\Gal(X_m/X)]$-module $\textup{Pr}(X_m)$ is generated by the elements
\begin{eqnarray*} p_{v,0} & = & -\sum_{e \in \textbf{E}_v(X_m)}(t(e)-v) \\ 
& = & -\sum_{e\in \textbf{E}_v(X),t(e)\neq v}(\alpha_m(e)t(e)-v)-\sum_{e\in \textbf{E}_v(X),t(e)=v}(\alpha_m(e)+\alpha_m(e)^{-1}-2)v,\end{eqnarray*} 
where $\textbf{E}_v(Y)$ denotes the set of edges with origin $v$ in a graph $Y$, and where $t(e)$ is the target of an edge $e$.

If $e$ is a loop with trivial voltage assignment then this term vanishes in the above sum. The coefficient of ${v = (v,1)}$ in $p_{v,0}$ (over $\Z_p$) is given by \[\deg(v)v-2\vert \{e\mid t(e)=v,\alpha_m(e)=1\}\vert.\] Recall that we just write $w$ instead of $(w,1)$ for a vertex of the base graph $X$. 

Now suppose that $X$ is a finite graph, and let ${n = |V(X)|}$. 
Fixing an indexing of the vertices of $X$, say ${V(X) = \{ v_1, \ldots, v_n\}}$, we will also write ${p_{i,0} = p_{v_i,0}}$ for these elements. We define a $\Z_p$-linear map
\[\mathcal{L}_\alpha\colon \textup{Div}(X_\infty)\longrightarrow  \textup{Pr}(X_\infty), \quad v_{i,0}\mapsto p_{i,0}.\]
It is easy to see from the definitions and the action of $\Gal(X_\infty/X)$ on $\textup{Div}(X_\infty)$ that this map is unique and well-defined. It is called the \emph{voltage $p$-Laplacian}, see also \cite{diss-gonet}. 

\begin{def1}
We say that a graph $X$ contains a non-trivial cycle of length $n>1$ if there exist a sequence of vertices $v_0,v_1,\dots v_{n-1},v_n,v_0$ and a sequence of edges $e_0,\dots e_{n}$ such that $e_i$ is an edge from $v_i$ to $v_{i+1}$ for $0\le i\le n-1$ and $e_n$ is an edge from $v_n$ to $v_0$. We assume furthermore that $e_{i+1}\neq {e_i}^{-1} $ for ${0\le i\le n-1}$ and $e_n\neq e_0^{-1}$.

For any such cycle $C$ we write $\beta_C$ for the group ring element $\prod_{i=0}^n\alpha(e_i)$.
\end{def1}
\begin{lemma} \label{lemma:connected} 
$X_m$ is connected for all $m$ if and only if we can find cycles $C_1,\dots , C_l$ such that $\{\beta_{C_1},\dots,\beta_{C_l} \}$ is a set of topological generators of $\Z_p^l$. 
\end{lemma}
\begin{proof}
  Recall that we write $V(X_m)$ for the set of vertices of $X_m$. We say that two vertices $(v,\gamma)$ and $(v',\gamma')$ of $X_m$ are connected if there is a path connecting $(v,\gamma)$ and $(v',\gamma')$. It is easy to check that this defines an equivalence relation on $V(X_m)$. 
  
 Assume first that there are cycles $C_1,\dots, C_l$ satisfying the above condition. In what follows, we abbreviate $\beta_{C_i}$ to $\beta_i$ for each ${1 \le i \le l}$. 
  Let $v\in V(X)$. We will first show that $(v,1)$ is connected to $(v,\gamma)$ for every $\gamma\in (\Z/p^m\Z)^l\setminus \{1\}$. As ${\beta_1,\dots ,\beta_l}$ are topological generators of $\Z_p^l$ it suffices to consider ${\gamma=\beta_i^{a_i}}$ for $1\le i\le l$ and $a_i\in \Z/p^m\Z$. Let $v_0$ be a vertex lying on $C_i$. Then $(v_0,\gamma')$ is connected to $(v_0,\gamma'\cdot \beta_i^{a_i})$ for every choice of $\gamma'$. As $X$ is connected we can find a path $e_0,e_1,\ldots, e_k$ from $v$ to $v_0$ and we see that $(v,1)$ is connected to ${(v_0,\prod_{j=1}^k\alpha_m(e_j))}$ and $(v_0,\tilde{\gamma})$ is connected to $(v,\tilde{\gamma}\prod_{j=1}^k\alpha_m(e^{-1}_j))$ for every choice of $\tilde{\gamma}$. We obtain the following relations
  \[\left(v,0\right)\sim \left(v_0,\prod_{j=1}^k \alpha_m(e_j)\right)\sim \left(v_0,\prod_{j=1}^k\alpha_m(e_j)\beta^{a_i}_i\right)\sim \left(v,\beta_i^{a_i}\right)\]
  proving our first claim. 
  
 For every $v'\in V(X)$ we can find an element $\gamma\in (\Z/p^m\Z)^l$ such that $(v,1)$ is connected to $(v',\gamma)$ -- using that $X$ is connected. Using the first part of the proof we see that $(v,1)\sim (v',\gamma)\sim (v',\gamma')$ for every choice of $\gamma'$. Therefore $X_m$ is indeed connected. 
 
 Assume now that all the $X_m$ are connected. In particular $X_1$ is connected. Let $v\in V(X)$ and $\gamma\in (\Z/p\Z)^l\setminus\{1\}$ be arbitrary. Then $(v,1)$ is connected to $(v,\gamma)$. Thus, there is a path $e_1,\dots, e_k$ from $v$ to $v$ such that $\gamma=\prod_{j=1}^k\alpha_1(e_j)$. Without loss of generality we can assume that $e_i\neq e^{-1}_{i+1}$ for $1\le i\le k-1$. Let $v_i$ be the end point of $e_i$ and set $v=v_0$. Let $k_v$ be the minimal index such that $e_{k_v}\neq e^{-1}_{k-k_v+1}$. Such an index exists, as ${\gamma\neq 1}$. Then $e_{k_v},\dots ,e_{k-k_v+1}$ is a path from $v_{k_v-1}$ to itself and it is indeed a cycle. Furthermore, ${\gamma=\prod_{j=k_v}^{k-k_v+1}\alpha_1(e_j)}$. 
  \end{proof}

\section{The module theoretic perspective in the multidimensional case} \label{section:J(X_m)} 
In this section we generalise the module-theoretic point of view from \cite{diss-gonet} to the case of arbitrary $l$: we will explain the structure of the Iwasawa module $\varprojlim_m J(X_m)$ and its relation to the asymptotic growth of $J(X_m)$.

Let $R=\Z_p[T_1,\ldots, T_l]$ and $\Lambda=\Z_p[[T_1,\ldots ,T_l]]$ be as before, and let $I_m$ and ${R_m = R/I_m}$ be defined as in Section~\ref{section:notation_iw-modules}. We write $\textup{Div}_\Lambda$, $\textup{Div}^0_\Lambda$ and $\textup{Pr}_\Lambda$ for $\textup{Div}(X_\infty)\otimes \Lambda$, $\textup{Div}^0(X_\infty)\otimes \Lambda$ and $\textup{Pr}(X_\infty)\otimes \Lambda$. It is easy to see that $\textup{Div}_\Lambda$ is $\Lambda$-free of rank $n$. Let $V$ be the submodule of $\textup{Div}_\Lambda$ generated by $p_{i,0}$, i.e. $V$ is the image of $\textup{Pr}_\Lambda$ in $\textup{Div}_\Lambda$. As $\textup{Div}(X_m)$ is generated by $v_{i,0}$ as an $R$-module, there is clearly a surjective map 
\[\textup{Div}_\Lambda/I_m\to \textup{Div}(X_m).\]
If there are elements $F_i\in \Lambda$ such that $\sum F_i \cdot p_{i,0}=0$ in $\textup{Div}_\Lambda$, then \[np^{ml}-1=\Z_p\textup{-rank}(\textup{Pr}(X_m))\le(n-1)p^{ml}+O(p^{m(l-1)})\] yielding a contradiction. Thus, $V$ is $\Lambda$-free as well and $\textup{Pr}_\Lambda\to \textup{Div}_\Lambda$ is injective.
Let $M_\Lambda$ be the submodule of $\textup{Div}_\Lambda$ generated by the elements $p_{i,0}$, the elements $v_{i,0}-v_{j,0}$ for ${1 \le j < i \le n}$ and $(T_1,\ldots, T_l)\textup{Div}_\Lambda$. The following theorem generalises results of \cite[Section~5]{diss-gonet}. 
\begin{thm} \label{thm:jacobian} 
We have $J(X_m)\cong M_\Lambda/((\omega_m(T_1),\ldots, \omega_m(T_l))\textup{Div}_\Lambda+\textup{Pr}_\Lambda)$.
\end{thm}
\begin{proof}
Let $\pi_m\colon \textup{Div}(X_\infty)\longrightarrow \textup{Div}(X_m)$ be the canonical map of voltage covers. Note that we can see this as an $R$-module homomorphism. We start by the following series of observations:
\begin{compactenum}[(a)]
    \item $\pi_m$ is surjective and induces surjective maps $\textup{Div}^0(X_\infty)\longrightarrow \textup{Div}^0(X_m)$ and ${\textup{Pr}(X_\infty)\longrightarrow \textup{Pr}(X_m)}$.
    \item The kernel of $\pi_m$ is given by $(\omega_m(T_1),\ldots, \omega_m(T_l))\textup{Div}_\Lambda\cap \textup{Div}(X_\infty)$.
    \item We have $(\omega_m(T_1),\ldots, \omega_m(T_l))\textup{Div}_\Lambda+\textup{Div}^0(X_\infty)=M_\Lambda$.
    \item $(\omega_m(T_1),\ldots, \omega_m(T_l))\textup{Div}_\Lambda+\textup{Pr}(X_\infty)=\textup{Pr}_\Lambda+(\omega_m(T_1),\ldots, \omega_m(T_l))\textup{Div}_\Lambda$. 
\end{compactenum}
Let us first see that this series of observations suffices to prove the theorem.
By points a) and b) we have an isomorphism
\[\textup{Div}^0(X_m)\cong \textup{Div}^0(X_\infty)/((\omega_m(T_1),\ldots, \omega_m(T_l))\textup{Div}_\Lambda\cap \textup{Div}^0(X_\infty)),\] where the right hand term is isomorphic to \[(\textup{Div}^0(X_\infty)+(\omega_m(T_1),\ldots, \omega_m(T_l))\textup{Div}_\Lambda)/(\omega_m(T_1),\ldots, \omega_m(T_l))\textup{Div}_\Lambda.\]
Combining this with c) and d), we have
\begin{align*}
    J(X_m)& = \textup{Div}^0(X_m) / \textup{Pr}(X_m) \\ 
    &=M_\Lambda/(\textup{Pr}_\Lambda+(\omega_m(T_1),\ldots, \omega_m(T_l))\textup{Div}_\Lambda)
\end{align*}
It remains to check the observations a)-d). In what follows we write ${|X| = n}$. For observation a) note that $v_{i,0}$, ${1\le i\le \vert X\vert}$, is a set of generators for $X_m$ (seen as an $R$-module). 
To see that it remains surjective when restricted to $\textup{Div}^0(X_\infty)$, it suffices to note that the elements ${v_{j,0}-v_{i,0}}$ and ${v_{j,\tau_k}-v_{j,0}}$ for a  set of topological generators ${\tau_1,\ldots, \tau_l}$ of $\Gal(X_\infty/X)$ generate $\textup{Div}^0(X_m)$ as well as $\textup{Div}^0(X_\infty)$ as $R$-modules. Using a similar argument for the generators $p_{i,0}$ instead shows the corresponding claim for $\textup{Pr}(X_\infty)$ and $\textup{Pr}(X_m)$.

For point b) note that $\textup{Div}_\Lambda$ is a free $\Lambda$-module of rank $\vert X\vert$ with generators $v_{j,0}$ for $1\le j\le \vert X\vert$. Furthermore, $\textup{Div}(X_m)$ is a free $R_m$-module with the same generators. As $\textup{Div}(X_\infty)$ contains the $R$-module generated by $v_{j,0}$, we see that $\textup{Div}(X_\infty)$ lies densely in $\textup{Div}_\Lambda$. We obtain the following isomorphisms, where we write ${n = \vert X \vert}$. 
\begin{align*}
\textup{Div}(X_m)&\cong \bigoplus_{i=1}^n R_m v_{i,0}\cong 
\bigoplus_{i=1}^n \Lambda/(\omega_m(T_1),\ldots ,\omega_m(T_l))\Lambda v_{i,0}\\ 
&\cong \textup{Div}_\Lambda/(\omega_m(T_1),\ldots , \omega_m(T_l))\\
&\cong \textup{Div}(X_\infty)/(\textup{Div}(X_\infty)\cap (\omega_m(T_1),\ldots,\omega_m(T_l))\textup{Div}_\Lambda),\end{align*} which proves claim b).

Now we prove c). By construction $\textup{Div}^0(X_\infty)\subseteq M_\Lambda$, and as our chosen generators of $M_\Lambda$ lie in $\textup{Div}^0(X_\infty)$, we see that $\textup{Div}^0(X_\infty)$ lies dense in $M_\Lambda$. Note that 
\[ (\omega_m(T_1),\ldots, \omega_m(T_l))\textup{Div}_\Lambda\] 
is a neighbourhood of $0$ in $M_\Lambda$. Let $y\in M_\Lambda$ be an arbitrary element, then there exists $z\in \textup{Div}^0(X_\infty)\cap (y+(\omega_m(T_1),\ldots, \omega_m(T_l))\textup{Div}_\Lambda)$. 
Thus, 
\begin{align*} y+(\omega_m(T_1),\ldots, \omega_m(T_l))\textup{Div}_\Lambda & \subseteq z+(\omega_m(T_1),\ldots, \omega_m(T_l))\textup{Div}_\Lambda \\ 
& \subseteq \textup{Div}^0(X_\infty)+(\omega_m(T_1),\ldots, \omega_m(T_l))\textup{Div}_\Lambda.\end{align*} 
Varying $y$ in $M_\Lambda$ gives the claim.

Claim d) can be proved similarly to point c) by using the generators $p_{i,0}$.
\end{proof}
\begin{lemma} \label{lemma:NundL} 
Let $L$ be a finitely generated $\Lambda$-torsion module and let $N\subseteq L$ be such that $L_m:=(\omega_m(T_1),\ldots, \omega_m(T_l))L\subseteq N$ for each $m$. Assume that $N/L_m$ is finite for all $m$. Then there are non-negative integers $m_0$ and $l_0$ such that
\begin{align} v_p(\vert N/L_m\vert) = m_0 p^{ml}+l_0 m p^{m(l-1)}+O(p^{m(l-1)}). \label{eq:star} \end{align} 
\end{lemma}
\begin{proof}
Since $N/L_m$ is finite for each $m$, the quotient $L_0/L_m$ is finite for all $m$. In particular, $L/L_m$ and $L/L_{0}$ have the same $\Z_p$-rank. 

To compute $v_p(\vert N/L_m \vert)$ note that we have
\[\vert N/L_m\vert = \vert N/L_0\vert \cdot \vert L_0/L_m\vert. \]
So it remains to compute $\vert L_0/L_m\vert$. 
Since $L/L_m$ and $L/L_0$ have the same $\Z_p$-rank and therefore the kernel of the natural surjective map ${L/L_m \longrightarrow L/L_0}$ is contained in the $p$-torsion of $L/L_m$, we deduce
\[\vert L_0/L_m\vert = \frac{\vert (L/L_m)[p^\infty]\vert }{\vert (L/L_0)[p^\infty] \vert}.\]  
The desired result follows directly from \cite[Theorems~3.4]{cuoco-monsky} (note that in our case we even have $\Z_p\textup{-rank}(L/L_m)=O(1)$ instead of $O(p^{m(l-2)})$).
\end{proof}

This lemma can be applied, according to Theorem~\ref{thm:jacobian}, to the following $\Lambda$-modules: consider ${N = M_\Lambda/\textup{Pr}_\Lambda}$ and ${L = \textup{Div}_\Lambda/\textup{Pr}_\Lambda}$. Then 
\[ L_m := (\omega_m(T_1), \ldots, \omega_m(T_l)) \cdot L \] 
is contained in $N$ for each ${m \in \N}$, and 
\[ N/L_m \cong J(X_m) \] 
is finite for each $m$ by Theorem~\ref{thm:jacobian}. 

In this case, we obtain the following arithmetic description of the numbers $m_0$ and $l_0$ from~\eqref{eq:star}: 
\begin{cor} \label{cor:zusammenhangzuN} 
If $l\ge 2$, then the parameters $m_0$ and $l_0$ from~\eqref{eq:star} are the generalised Iwasawa invariants of $L$. Since ${L/N \cong \Z_p}$ is pseudo-null over $\Lambda$, these are also the generalised Iwasawa invariants of $N$. 

If $l=1$, then we obtain the Iwasawa invariants of $N$. In particular, $m_0$ equals the $\mu$-invariant of $L$, while ${l_0=\lambda(N) = \lambda(L)-1}$ in this case.
\end{cor}
\begin{proof}
If $l\ge 2$ this follows directly from the above proof. In the case $l=1$ it is a straight forward computation over one-dimensional Iwasawa algebras (note that ${\varprojlim_m N/I_m \cong N}$, since ${\cap_m I_m = \{0\}}$).
\end{proof}

\section{Computation of the order of $J(X_m)$} \label{section:linearalgebra} 
Let $X$ be a finite graph, and let $\alpha \colon S\to \Z_p^l$ be as in Section~\ref{section:notation}.
Recall the definition of the $\Z_p$-linear map $\mathcal{L}_\alpha$ from Section~\ref{section:notation}. We can extend naturally the linear map $\mathcal{L}_\alpha$ to a $\Lambda$-linear map.
Recall that $\textup{Div}_\Lambda\cong \Lambda^n$ with basis $\{ v_{1,0}, \ldots, v_{n,0}\}$. We would like to give a matrix-representation of $\mathcal{L}_\alpha$ with respect to this basis. Let $D$ be the degree matrix of the base graph $X$ and define the voltage assignment matrix $A_\alpha=(\alpha_{i,j})_{1\le i,j\le n}$ as follows: 
\[\alpha_{i,j}=\begin{cases}
    \sum _{\{ e(v_i,v_j)\in E(v_i,v_j)\}}\alpha(e(v_i,v_j)) \quad & \textup{if } i\neq j\\
    \sum_{e(v_i,v_i)\in E(v_i,v_i)\}}\alpha(e(v_i,v_i))+\alpha^{-1}(e(v_i,v_i)) \quad & \textup{if } i=j
\end{cases}.\] 
Then $\alpha_{i,j} \in \Lambda$ for all $i,j$. It is easy to verify that $\mathcal{L}_\alpha$ is represented by the matrix $\Delta_\infty=(D-A_\alpha)^t$ (here we use the $v_{i,0}$ as a basis, and we denote by $M^t$ the transpose of a matrix $M$). Recall that 
\[p_{v_i,0}=-\sum_{e\in E_{v_i}(X_\infty)}(\alpha(e)t(e)-v_i)=\deg(v_i)v_i-\sum_{j=1}^n \alpha_{i,j}v_j,\]
which is represented by the $i$-th row of $D-A_\alpha$ and not by the $i$-th column. Note that this ambiguity does not occur if we understand $\mathcal{L}_\alpha$ only as a $\Z_p$-linear map. In this case a $\Z_p$-basis of $\textup{Div}(X_m)$ is $\{(v_i,g)\mid v_i\in V(X), g\in (\Z/p^m\Z)^l\}$,  $\mathcal{L}_\alpha$ is represented by a symmetric matrix and we do not have to consider the transpose.

In what follows, we fix some ${m \in \N}$. Recall the definition of ${R_m = R/I_m}$ from Section~\ref{section:notation_iw-modules}. If we define $\Delta^t_m$ to be the representing matrix for the map 
\[ \mathcal{L}_\alpha \colon \textup{Div}(X_m) \longrightarrow \textup{Div}(X_m) \]  
(again in the $R_m$-basis  $v_{i,0}$), we obtain that $\Delta_m$ is the image of $\Delta_\infty$ in $\textup{Mat}_{n,n}(R_m)$.

Let $\Omega$ be the group of all $p$-power roots of unity in some fixed algebraic closure $\overline{\Q_p}$ of $\Q_p$. We say that two elements in $\Omega^l$ are equivalent if they are Galois conjugate to each other.
Using that $\textup{Div}_\Lambda$ is $\Lambda$-free of rank $n$, we can define 
\[\phi_m\colon \textup{Div}_\Lambda/I_m\longrightarrow (\bigoplus \Z_p[\zeta])^n, \] 
where each component map is defined as in \cite{cuoco-monsky}. Here the sum on the right hand side runs over all equivalence classes of $\Omega^l[p^m]$.
Let $\tilde{\Lambda}_m$ be the image of $\phi_m$. Note that by \cite[Lemma 2.1]{cuoco-monsky} $\phi_m$ has a finite cokernel. Recall that $\textup{Pr}_\Lambda=\mathcal{L}_\alpha(\textup{Div}_\Lambda)$ and that $\Delta^t_\infty$ is a matrix representing the map $\mathcal{L}_\alpha$. 

It is immediate that $\mathcal{L}_\alpha(I_m\textup{Div}_\Lambda)\subseteq I_m\textup{Div}_\Lambda$. So $\mathcal{L}_\alpha$ induces well-defined maps on $\textup{Div}_\Lambda/I_m$. By abuse of notation, we denote these maps by $\mathcal{L}_\alpha$ again, and we define a map $\tilde{\mathcal{L}_\alpha}$ on $\tilde{\Lambda}_m$ such that
\[\phi_m\circ \mathcal{L}_\alpha=\tilde{\mathcal{L}_\alpha}\circ \phi_m\] 
on $\textup{Div}_\Lambda/I_m$. 
Note that $\tilde{\mathcal{L}_\alpha}$ is a linear map of $\Z_p$-modules. 
Recall that $\phi_m$ has finite cokernel. So we can extend $\tilde{\mathcal{L}_\alpha}$ to a linear map of $\tilde{\Lambda}_m\otimes \Q_p=(\oplus_{z=1}^k\Q_p[\zeta_z])^n$, where $\zeta_1, \ldots, \zeta_k$ represent the different equivalence classes in $\Omega^l[p^m]$ modulo Galois equivalence.

For every matrix $A$ with coefficients ${a_{i,j} \in \Lambda}$ and each ${\zeta \in \Omega^l}$ we denote by $A(\zeta - 1)$ the matrix with coefficients ${a_{i,j}(\zeta-1) \in \Z_p[\zeta]}$ (i.e. we replace $T_i$ by the $i$-th component of $\zeta - 1$). 
\begin{lemma}
  We can choose a basis of $\tilde{\Lambda}_m\otimes \Q_p$ such that $\tilde{\mathcal{L}_\alpha}$ is represented by a $p^{ml}n\times p^{ml}n$-block matrix.
 \[A_m= \begin{pmatrix}
  \Delta^t_\infty(\zeta_1-1)&0&\dots &\dots&\dots&0\\
 0& \Delta^t_\infty(\zeta_1-1)&\dots &\dots&\dots&0\\
  \dots&\dots&\dots&\dots&\dots&\dots\\
  0&0&\dots&\Delta^t_\infty(\zeta_2-1)&\dots&0\\
  \dots&\dots&\dots&\dots&\dots&\dots\\
  0&\dots&\dots&\dots&0&\Delta^t_\infty(\zeta_{k}-1)
  \end{pmatrix},\]
  where $\zeta_1,\ldots ,\zeta_k$ represent the different equivalence classes in $\Omega^l[p^m]$ modulo Galois equivalence and each block $\Delta_\infty^t(\zeta_i-1)$ occurs $\varphi(\textup{ord}(\zeta_i))$ times, where $\varphi$ denotes Euler's totient function. 
\end{lemma}
\begin{proof}
  Let $F$ be an element in $\textup{Div}_\Lambda$. 
  Note that $\phi_m\circ \mathcal{L}_\alpha(F)$ is a vector with components $\Delta^t_\infty(\zeta_i-1)F(\zeta_i-1)$. 
  
We fix an index $j$ ($1\le j\le k$) and let $F_jv_{i,0}\in \Lambda v_{i,0}$ be such that $\phi_m(F_jv_{i,0})$ vanishes at all $\zeta_z-1$ for $z\neq j$ and is an element $a\in \Z_p\setminus\{0\}$ at $\zeta_j-1$. Let $G$ be an arbitrary element in $\Lambda$. Then we have
  \[\tilde{\mathcal{L}_\alpha}(\phi_m(GF_jv_{i,0}))=\phi_m(\mathcal{L}_\alpha (GF_jv_{0, i}))=\phi_m(GF_j \Delta^t_\infty v_{i,0}).\] 
  Recall that $G(\zeta-1)F_j(\zeta-1)=0$ for all $\zeta$ that do not lie in the equivalence class of $\zeta_j$. Thus
  \[G(\zeta-1)F_j(\zeta-1)\Delta^t_\infty(\zeta-1)=0\]
  for all these $\zeta$. In particular, $\phi_m(GF_j\Delta^t_\infty v_{i,0})$ has only one non-trivial component, namely at $\zeta_j-1$. Note that $\phi_m((\Lambda/I_m) F_jv_{i,0})$ has finite index in $\Z_p[\zeta_j]$. 
  Using the fact that $\bigoplus_{z=1}^k\Z_p[\zeta_z]$ is $\Z_p$-free, we have shown that if we start with an element in $\tilde{\Lambda}_m$ that is non-trivial at exactly one $\zeta_j-1$, then the same holds for its image under $\tilde{\mathcal{L}_\alpha}$.
  From this we can deduce that $\tilde{\mathcal{L}_\alpha}(\Q_p[\zeta_j])^n\subseteq (\Q_p[\zeta_j])^n$.
  
  It remains to compute the corresponding matrix. Let $G_1,\ldots ,G_{\varphi(\textup{ord}(\zeta_j))-1}$ be polynomials such that $G_w(\zeta_j-1)=\zeta_j^w$ and let $G_0=1$. Note that $\phi_m(G_wF_jv_{i,0})$ for $0\le w\le \varphi(\textup{ord}(\zeta_j))-1$ spans a sublattice of finite index in $\Z_p[\zeta_j]$. Using the above computations (with $G_w$ for $G$) we see that
  \[\tilde{\mathcal{L}_\alpha}(\zeta_j^wa\phi_m(v_{i,0}))=\sum_{u=1}^n\Delta^t_\infty(\zeta_j-1)\zeta_j^wa\phi_m(v_{u,0}).\] 
  Therefore, we see that the block $\Delta^t_\infty(\zeta_j-1)$ has to occur $\varphi(\textup{ord}(\zeta_j))$ times in the matrix representing $\tilde{\mathcal{L}_\alpha}$ on $\tilde{\Lambda}_m\otimes \Q_p$. 
\end{proof}
\begin{lemma}
  \[v_p(\vert J(X_m)\vert)=\sum_{\zeta\neq 1}v_p(\det(\Delta^t_\infty(\zeta- 1)))+v_p(|J(X_0)|). \]
\end{lemma}
\begin{proof}
Recall that $N/I_mL$ is finite for all $m$ and that $L/N$ is annihilated by $(T_1,\ldots, T_l)$ and isomorphic to $\Z_p$. It follows that $L/I_mL$ has $\Z_p$-rank one. By the definition of $L$ we see that \[L/I_m L=\textup{Div}_\Lambda/(I_m+\textup{Pr}_\Lambda)=\textup{Div}(X_m)/\textup{Pr}(X_m)=\textup{Div}(X_m)/\Delta^t_m\textup{Div}(X_m).\] 
Therefore the kernel of the linear map given by multiplication with $\Delta^t_m$ has $\Z_p$-rank 1. Since $\phi_m$ is an injection, it follows that $A_m$ has rank $p^{ml}n-1$ for all $m$. In particular, the matrix $\Delta^t_\infty(0)$ has determinant zero, but all the matrices $\Delta^t_\infty(\zeta_i-1)$ have non-zero determinant. 

Fix a basis of $\tilde{\Lambda}_m$ and let $A'_m$ be the matrix describing $\tilde{\mathcal{L}_\alpha}$ in this basis. Then $A_m$ and $A'_m$ are conjugate. 
Moreover, $\textup{Div}(X_m)$ is a free $\Z_p$-module, and 
\[ \textup{Div}(X_m)/\Delta^t_m \textup{Div}(X_m) \cong \tilde{\Lambda}_m/A'_m \tilde{\Lambda}_m. \] 
We are interested in $(\textup{Div}(X_m)/\Delta^t_m\textup{Div}(X_m))[p^\infty]$. Let $B_m$ and $B'_m$ be the matrices representing 
\[\mathcal{L}_\alpha \colon (\textup{Div}_\Lambda/I_m \otimes \Q_p)/\ker(A_m)\longrightarrow (\textup{Div}_\Lambda/I_m \otimes \Q_p)/\ker(A_m)\]

and 
\[\tilde{\mathcal{L}_\alpha}\colon \tilde{\Lambda}_m/\ker(A'_m)\longrightarrow \tilde{\Lambda}_m/\ker(A'_m)\]
  in our chosen basis. Then $B_m$ and $B'_m$ are conjugate and we obtain 
 \[\vert(\textup{Div}(X_m)/\Delta^t_m\textup{Div}(X_m))[p^\infty]\vert=v_p(\det(B'_m))=v_p(\det(B_m)).\]
 
 We let $\zeta_k=1$. Then $B_m$ has the form
 \[\begin{pmatrix}
  \Delta^t_\infty(\zeta_1-1)&0&\dots &\dots&\dots&0\\
 0& \Delta^t_\infty(\zeta_1-1)&\dots &\dots&\dots&0\\
  \dots&\dots&\dots&\dots&\dots&\dots\\
  0&0&\dots&\Delta^t_\infty(\zeta_2-1)&\dots&0\\
  \dots&\dots&\dots&\dots&\dots&\dots\\
  0&\dots&\dots&\dots&0&B_0
  \end{pmatrix},\]
 where $v_p(\det(B_0))= v_p(\vert J(X_0)\vert)$.
Using the structure of the matrix $B_m$, the claim follows. 
\end{proof}

In particular, this proves Greenberg's conjecture, i.e. we obtain the following 
\begin{thm}
\label{thm:greenberg-conj}
    Greenberg's conjecture holds true for the growth of $v_p(|J(X_m)|)$, i.e. there exists a polynomial ${Q(X,Y) \in \Z_p[X,Y]}$ of total degree $l$ and degree at most one in the second variable such that 
    \[ v_p(|J(X_m)|) = Q(p^m,m) \] 
    for all sufficiently large ${m \in \N}$. 
\end{thm}
\begin{rem}
Our method gives the same polynomial as the method using Ihara L-functions from \cite[Theorem~6.1]{dubose-vallieres}.
\end{rem}

\section{An Iwasawa main conjecture} \label{section:main_conjecture} 
Let as before $X$ be a finite connected multigraph. Let $X_\infty$ be a $\Z_p^l$-voltage cover of $X$. For every finite graph we define the zeta-function and the Ihara $L$-functions as in \cite{dubose-vallieres}.
Let $\Delta_\infty$ be defined as in Section~\ref{section:linearalgebra}. Let $\psi$ be a character of $\Z_p^l$ with finite image and let $A_\psi$ be the matrix defined in \cite{dubose-vallieres}. From the definitions in \cite{dubose-vallieres} we see that 
\[P_{\psi}(u):=\det(I-A_\psi u+(D-I)u^2)=\frac{(1-u^2)^{\chi(X)}}{L_X(u,\psi)}.\]
In this section we assume that no vertex has valency $1$. If $\psi$ is not the trivial character then $P_\psi(1)\neq 0$. Using the definition of $A_\psi$, it follows that
\[P_\psi(1)=\psi(\det(\Delta_\infty)).\]
If we interpret the function $\lim_{u\to 1}\frac{(1-u^2)^{\chi(X)}}{L_X(u,\psi)}$ as the algebraic part of the Ihara $L$-function, then $\det(\Delta_\infty)\in \Lambda$ can be seen as a $p$-adic $L$-series interpolating special values of the algebraic part of $L_X(u,\psi)$. 

\begin{lemma}
\label{lemma:projective-limit}
Let $X_m$ be the level $m$ subcover of $X_\infty/X$. Then we have
\[J_\infty:=J(X_\infty)\otimes \Lambda=\varprojlim_{m}J(X_m) \cong \varprojlim_m N/(I_m L) \cong N.\]
\end{lemma} 
\begin{proof}
By definition $J(X_\infty)=\textup{Div}^0(X_\infty)/\textup{Pr}(X_\infty)$ as $R$-modules. Consider the exact sequence
\begin{equation}\label{trivial-sequence}0\longrightarrow \textup{Pr}(X_\infty)\longrightarrow \textup{Div}^0(X_\infty)\longrightarrow J(X_\infty)\longrightarrow 0.\end{equation}
Recall that $\textup{Div}^0(X_\infty)+I_m\textup{Div}_\Lambda=M_\Lambda$ for all $m$. It follows that the image of $\textup{Div}^0_\Lambda$ in $\textup{Div}_\Lambda$ is equal to $M_\Lambda$. Note that $\textup{Div}(X_\infty)/\textup{Div}^0(X_\infty)$ is annihilated by $J_0=(T_1,T_2,\dots, T_l)$. But $\textup{Div}^0(X_m)[J_0]=(\prod_{i=1}^l\nu_{m,i}(T_i))\textup{Div}(X_m)\cap \textup{Div}^0(X_m)$.
It follows that $\textup{Div}^0_\Lambda$ does not contain non-trivial $J_0$-torsion. In particular, the image of $\textup{Tor}_1^R(\textup{Div}(X_\infty)/\textup{Div}^0(X_\infty))$ inside $\textup{Div}^0_\Lambda$ is trivial and we obtain a natural injection 
\[\textup{Div}^0_\Lambda\longrightarrow \textup{Div}_\Lambda.\]
Thus, $\textup{Div}^0_\Lambda\cong M_\Lambda$. 

Upon tensoring \eqref{trivial-sequence} with $\Lambda$ we obtain an exact sequence
\[\textup{Pr}_\Lambda\longrightarrow M_\Lambda\longrightarrow J_\infty \longrightarrow 0.\]
Recall that $\textup{Pr}_\Lambda=\mathcal{L}_\alpha(\textup{Div}_\Lambda)\subset M_\Lambda$ by definition. In particular, the natural map $\textup{Pr}_\Lambda\longrightarrow M_\Lambda$ is injective and we obtain an isomorphism
\begin{equation}
\label{trivial-isomorphism}
    M_\Lambda/\textup{Pr}_\Lambda=J_\infty.
\end{equation}

Taking projective limits of the exact sequences
\[0\longrightarrow \textup{Pr}(X_m)\longrightarrow \textup{Div}^0(X_m)\longrightarrow J(X_m)\longrightarrow 0\] gives us a sequence
\[0\longrightarrow \varprojlim_m\textup{Pr}(X_m)\longrightarrow \varprojlim_m\textup{Div}^0(X_m)\longrightarrow \varprojlim_mJ(X_m)\]
By Theorem \ref{thm:jacobian} $\varprojlim_mJ(X_m)\cong N$. Note that $\textup{Pr}(X_m)$ is generated by $p_{i,0}$ and that these generators do not have a relation over $\Lambda$. It follows that $\varprojlim_m\textup{Pr}(X_m)\cong \textup{Pr}_\Lambda$. The module $\textup{Div}^0(X_m)$ is generated by $J_0\textup{Div}^0(X_m)$, $v_{i,0}-v_{j,0}$ for $1\le i<j\le n$ and the elements $p_{i,0}$. It follows that $\varprojlim_m\textup{Div}^0(X_m)$ is generated by the same elements over $\Lambda$ and we obtain that $\varprojlim_m\textup{Div}^0(X_m)\cong M_\Lambda$. Summarising, we obtain an exact sequence
\[0\longrightarrow \textup{Pr}_\Lambda\longrightarrow M_\Lambda\longrightarrow \varprojlim_mJ(X_m)\longrightarrow 0\]
Together with \eqref{trivial-isomorphism} the claim follows.
\end{proof}
Fixing this notation, we can prove an \emph{Iwasawa main conjecture} for graphs: 
\begin{thm} \label{thm:main_conjecture} 
Let $l\ge 2$. Then 
\[\textup{Char}(J_\infty)=(\det(\Delta_\infty)). \] 
\end{thm}
\begin{proof}
Recall that $\Delta^t_\infty$ is the matrix representing the $\Lambda$-linear map 
\[\mathcal{L}_\alpha \colon \textup{Div}_\Lambda \longrightarrow \textup{Div}_\Lambda, \]
such that $\mathcal{L}_\alpha(\textup{Div}_\Lambda)=\textup{Pr}_\Lambda$. Moreover,  $\textup{Div}(X_\infty)/\textup{Pr}(X_\infty)$ is a finitely generated torsion $R=\Z_p[T_1,\ldots, T_l]$-module and  ${\textup{Pr}_\Lambda = \textup{Pr}(X_\infty) \otimes \Lambda}$ is a free $\Lambda$-module. Consider the natural exact sequence
\[\textup{Pr}_\Lambda\longrightarrow \textup{Div}_\Lambda\longrightarrow (\textup{Div}(X_\infty)/\textup{Pr}(X_\infty))\otimes \Lambda\longrightarrow 0.\]

We would like to see that this sequence is still left exact. Note that 
\[ \textup{Tor}_1^R(\textup{Div}(X_\infty)/\textup{Pr}(X_\infty),\Lambda)\] 
is still $\Lambda$-torsion. On the other hand, $\textup{Pr}_\Lambda$ is $\Lambda$-free. Thus, the image of \[\textup{Tor}_1^R(\textup{Div}(X_\infty)/\textup{Pr}(X_\infty),\Lambda)\] 
in $\textup{Pr}_\Lambda$ is trivial and the above sequence is still left exact.

As $\textup{Div}_\Lambda$ and $\textup{Pr}_\Lambda$ are free of the same $\Lambda$-rank we obtain
\[(\det(\Delta_\infty))=\textup{Fitt}^0_\Lambda((\textup{Div}(X_\infty)/\textup{Pr}(X_\infty))\otimes \Lambda)=\textup{Char}((\textup{Div}(X_\infty)/\textup{Pr}(X_\infty))\otimes \Lambda).\]

Analogously to the first exact sequence we can show that
\[0\longrightarrow \textup{Pr}_\Lambda\longrightarrow \textup{Div}^0_\Lambda\longrightarrow J_\infty\longrightarrow 0\]
is exact.

There is a natural injection
\[\psi\colon J_\infty\longrightarrow (\textup{Div}_\Lambda/\textup{Pr}_\Lambda)\otimes \Lambda,\]
whose cokernel has $\Z_p$-rank one. In particular
\[\textup{Char}((\textup{Div}(X_\infty)/\textup{Pr}(X_\infty))\otimes \Lambda)=\textup{Char}(J_\infty), \] 
since ${l \ge 2}$ and therefore each finitely generated $\Z_p$-module is pseudo-null over $\Lambda$. 
\end{proof}
\begin{rem} \label{rem:main_conjecture} 
Note that this proof crucially depends on the fact that $l\ge 2$. In the case $l=1$, we see that ${\textup{Char}(\coker \psi)=(T)}$. Thus, in this case
\[(T) \cdot \textup{Char}(J_\infty)= (\det(\Delta_\infty)).\]
\end{rem}

\section{Fukuda's theorem and local boundedness results} 
\label{sec:Fukuda}
We start from the $\Lambda$-module isomorphisms 
\[ J(X_m) \cong N/L_m \] 
from Theorem~\ref{thm:jacobian}. Let $\mathfrak{m} = (p, T_1, \ldots, T_l)$ be the maximal ideal of the Iwasawa algebra $\Lambda$. For any ideal ${I \subseteq \mathfrak{m}}$ of $\Lambda$ and any finitely generated torsion $\Lambda$-module $A$, we define 
\[ \rg_I(A) = v_p(|A/(I \cdot A)|), \] 
whenever this is finite. 

The following key result is proven as \cite[Theorem~2.5]{local_beh}. 
\begin{thm} \label{thm:fukuda} 
  Suppose that there exists an ideal $I \subseteq \Lambda$ such that 
  \[ \rg_I(J(X_m)) = \rg_I(J(X_{m+1})) \] 
  for some ${m \in \N}$. Then 
  \[ \rg_I(J(X_k)) = \rg_I(J(X_m)) \] 
  for every ${k \ge m}$, and in fact $\rg_I(N)$ is finite and equal to $\rg_I(J(X_m))$. 
\end{thm} 
\begin{proof} 
  Since $N/(I \cdot N + L_m)$ and $N/(I \cdot N + L_{m+1})$ have the same cardinality, we may conclude that 
  \[ I \cdot N + L_m = I \cdot N + L_{m+1}. \] 
  Letting $Z := (I \cdot N + L_m)/(I \cdot N)$ and recalling the definition of $L_m$, it follows that 
  \[ Z \subseteq \mathfrak{m} \cdot Z. \] 
  Since $Z$ is a compact $\Lambda$-module, Nakayama's Lemma implies that ${Z = 0}$, i.e. $L_m$ is contained in $I \cdot N$. The assertion of the theorem follows immediately. 
\end{proof} 
\begin{lemma} \label{lemma:no-pseudo-null} 
Let $N=M_\Lambda/\textup{Pr}_\Lambda$. Then $N \subseteq L=\textup{Div}_\Lambda/\textup{Pr}_\Lambda$. The two $\Lambda$-modules $N$ and $L$ do not contain any non-zero pseudo-null submodules. 
\end{lemma}
\begin{proof}
Note that $\textup{Div}_\Lambda$ and $\textup{Pr}_\Lambda$ are both free $\Lambda$-modules of rank $n$. Suppose that $x\in \textup{Div}_\Lambda$ is an element such that $x+\textup{Pr}_\Lambda$ generates a pseudo-null submodule in $L$. Choose two coprime elements $a$ and $c$ such that $ax,cx\in \textup{Pr}_\Lambda$. Fix a basis $b_1, \ldots,  b_n$ of $\textup{Pr}_\Lambda$ and choose coefficients
\[a x=\sum_{i=1}^n a_i b_i, \quad cx=\sum_{i=1}^n c_ib_i.\]
Note that the coefficients $a_i$ and $c_i$ are unique. It follows that
\[ca_i=ac_i \textup{ for } 1\le i\le n.\]
As $(c,a)=1$, it follows that $a\mid a_i$ for $1\le i\le n$.

 Therefore we can write 
\[ a x = \sum_{i = 1}^n a \cdot \frac{a_i}{a} \cdot b_i = a \cdot \sum_{i = 1}^n \frac{a_i}{a} \cdot b_i, \] 
and thus 
\[ a \cdot \left(x - \sum_{i = 1}^n \frac{a_i}{a} \cdot b_i \right) = 0. \] 
Since $\textup{Div}_\Lambda$ is a free $\Lambda$-module, this is possible only if already 
\[ x = \sum_{i = 1}^n \frac{a_i}{a} \cdot b_i \] 
is contained in $\textup{Pr}_\Lambda$. Therefore $L$ (and $N$) do not contain any non-trivial pseudo-null submodule. 
\end{proof}
The property proved above has many important consequences. For example, since $L$ does not contain any non-trivial pseudo-null submodule, we obtain, starting from a pseudo-isomorphism, an \emph{injection} ${\psi \colon L \longrightarrow E}$ into some elementary $\Lambda$-module $E$ such that the cokernel of $\psi$ is pseudo-null.

Now we introduce a topology on the set of $\Z_p^l$-covers of our fixed finite multigraph $X$. The voltage graphs $X_\infty$ and $\tilde{X}_\infty$ assigned to two voltage covers 
\[ \alpha, \tilde{\alpha} \colon S \longrightarrow \Z_p^l \] 
will be considered as \lq close' if the induced maps 
\[ \alpha_m, \tilde{\alpha}_m \colon S \longrightarrow (\Z/p^m \Z)^l \] 
coincide for a large integer $m$. More formally, let $\mathcal{E}^l(X)$ denote the set of voltage $\Z_p^l$-covers of $X$. For any voltage cover ${X_\infty \in \mathcal{E}^l(X)}$ and each ${m \in \N}$, we denote by $\mathcal{E}(X_\infty,m)$ the set of voltage covers $\tilde{X}_\infty$ which satisfy the above condition (i.e. ${\alpha_m = \tilde{\alpha}_m}$). It follows from our definition of the voltage cover graphs $X_m$ (see Definition~\ref{def:voltagelevel}) that 
\[ \tilde{X}_n = X_n \] 
for each ${n \le m}$ if ${\tilde{X}_\infty \in \mathcal{E}(K_\infty, m)}$. 

The sets $\mathcal{E}(X_\infty,m)$, with ${X_\infty \in \mathcal{E}^l(X)}$ and ${m \in \N}$, form the basis of a topology on $\mathcal{E}^l(X)$. This topology is motivated by what we call \emph{Greenberg's topology} on the set $\mathcal{E}^l(K)$ of $\Z_p^l$-extensions of a number field $K$ (the latter topology has been introduced in this classical setting in \cite{green_73}).
\begin{lemma} \label{lemma:top_kompakt} 
  The set $\mathcal{E}^l(X)$ is compact with respect to Greenberg's topology. 
\end{lemma} 
\begin{proof} 
  Fix a finite graph $X$ and a voltage cover 
  \[ \alpha \colon S \longrightarrow \Z_p^l. \] 
  For every ${m \in \N}$, the set $\mathcal{E}^{l,m}(S)$ of maps 
  \[ \alpha_m \colon S \longrightarrow (\Z/p^m\Z)^l \] 
  is finite. Therefore $\mathcal{E}^{l,m}(S)$ is compact (with regard to the discrete topology) for each $m$. Since ${\mathcal{E}^l(X) = \varprojlim_m \mathcal{E}^{l,m}(S)}$, it follows that $\mathcal{E}^l(X)$ is also compact. 
\end{proof} 

In the following, we want to compare the (generalised) Iwasawa invariants of voltage covers ${X_\infty, \tilde{X}_\infty \in \mathcal{E}^l(X)}$ which are close with respect to Greenberg's topology. We start with the case ${l \ge 2}$. 
\begin{thm} \label{thm:local_max1} 
   Fix an element ${X_\infty \in \mathcal{E}^l(X)}$, ${l \ge 2}$. Then there exist integers ${i, k \in \N}$ such that with ${U := \mathcal{E}(X_\infty,i)}$ the following two statements hold for each ${\tilde{X}_\infty}$ in $U$. \begin{compactenum}[(a)] 
     \item $m_0(\tilde{X}_\infty) \le m_0(X_\infty)$, and 
     \item $l_0(\tilde{X}_\infty) \le k$ holds if ${m_0(\tilde{X}_\infty) = m_0(X_\infty)}$. 
   \end{compactenum} 
\end{thm} 
\begin{proof} 
  Recall that we want to study the generalised Iwasawa invariants of the $\Lambda$-module ${N = M_\Lambda/\textup{Pr}_\Lambda}$ (see also Corollary~\ref{cor:zusammenhangzuN}). Let ${f_N \in \Lambda}$ be the characteristic power series of ${N}$. Recall that ${\Lambda = \Z_p[[T_1, \ldots, T_l]]}$. We say that $f_N$ is \emph{in Weierstraß normal form with respect to $T_l$} if $f_N$ is associated with a product ${p^{m_0(N)} \cdot g_N}$ for some element $g_N$ of the form 
  \begin{align} \label{eq:weierstrass} g_N = T_l^k + h_{k-1} T_l^{k-1} + \ldots + h_1 \cdot T_l + h_0, \end{align} 
  where ${k \ge 1}$ and ${h_0, \ldots, h_{k-1} \in (p, T_1, \ldots, T_{l-1}) \subseteq \Z_p[[T_1, \ldots, T_{l-1}]]}$. It follows from \cite[Lemma~1]{babaichev} that given any set ${\gamma_1, \ldots, \gamma_l}$ of topological generators of 
  \[ \Gal(X_\infty/K) \cong \Z_p^l, \] 
  we can change this set of generators to a set of generators ${\gamma_1', \ldots, \gamma_{l-1}', \gamma_l}$ (i.e. without changing the last element) such that with respect to these new generators, $f_N$ is in Weierstraß normal form with respect to the last variable $T_l$. 
  
  Now let $E_N$ be an elementary $\Lambda$-module attached to $N$, and fix a pseudo-isomorphism ${\varphi: E_N \longrightarrow N}$. Then the cokernel ${B := N/\varphi(E_N)}$ is a pseudo-null $\Lambda$-module, and it follows from \cite[Lemma~4.4]{local_max} that the topological generators ${\gamma_1, \ldots, \gamma_l}$ of ${\Gal(X_\infty/X) \cong \Z_p^l}$ can be chosen such that \begin{compactitem} 
    \item $\gamma_l = T_l - 1$ remains unchanged, $f_N$ is still in Weierstraß normal form with respect to $T_l$, and the degree $k$ in the representation~\eqref{eq:weierstrass} of $g_N$ does not change, and 
    \item $B/(T_1, \ldots, T_{l-2}, T_{l-1} - p^x)$ is finite for all but finitely many ${x \in \N}$. 
  \end{compactitem} 
  
  Fix such an integer $x$. It follows from \cite[Lemma~4.5]{local_max} that 
  \[ \rg_n(N) := v_p(|N/(T_1, \ldots, T_{l-2}, T_{l-1} - p^x, \nu_{2n,n}(T_l)|) \] 
  is finite for each sufficiently large ${n \in \N}$. Moreover, it follows from the short exact sequence 
  \[ \xymatrix{0 \ar[r] & E_N \ar[r]^\varphi & N \ar[r] & B \ar[r] & 0} \] 
  that 
  \[ \rg_n(N) \le \rg_n(E_N) + \rg_n(B) \] 
  (in particular, all the terms in this inequality are finite; cf. also the proof of \cite[Theorem~3.2]{local_max}). Note that 
  \[ \rg_n(B) = v_p(|B/(T_1, \ldots, T_{l-2}, T_{l-1} - p^x)|) =: C \] 
  for each sufficiently large ${n \in \N}$ (it will suffice if $n > C$), i.e. $\rg_n(B)$ is bounded in $n$. 
  
 From now on we will consider $n$ large enough to ensure that \begin{compactenum}[(a)] 
   \item $\rg_n(N)$ is finite, 
   \item $n > C$, and 
   \item ${p^{n}(p-1) > n \cdot k + C}$, where $k$ is as in~\eqref{eq:weierstrass}. 
 \end{compactenum} 
 Then it follows from the proof of \cite[Theorem~4.6]{local_max} (cf. in particular equation~(4.4)) that 
 \[ \rg_n(E_N) = m_0(X_\infty) \cdot (p^{2n} - p^n) + n \cdot k. \] 
 Since ${J(X_m) \cong N/L_m}$ and therefore ${\rg_n(J(X_m)) \le \rg_n(N)}$ for each ${m \in \N}$, it follows that we can choose $m$ sufficiently large such that 
 \[ \rg_n(J(X_m)) = \rg_n(J(X_k)) = \rg_n(N) \] 
 for each ${k \ge m}$ (cf. also the proof of Theorem~\ref{thm:fukuda}). Fix such an integer $m$, and consider ${U := \mathcal{E}^l(X_\infty, m + 1)}$. Note that the integer $m$ depends on $n$; this is not a problem since $n$ has been chosen and fixed above. 
 
 Let ${\tilde{X}_\infty \in U}$ be arbitrary. We denote the corresponding $\Lambda$-module $N(\tilde{X}_\infty)$ by $\tilde{N}$. 
 Since ${\tilde{X}_\infty \in U}$, we have 
 \[ \rg_n(J(\tilde{X}_m)) = \rg_n(J(\tilde{X}_{m+1})) = \rg_n(N). \] 
 Therefore Theorem~\ref{thm:fukuda} implies that 
 \[ \rg_n(\tilde{N}) = \rg_n(N). \] 
 Moreover, it follows from \cite[Theorem~3.2]{local_max} that 
 \[ \rg_n(E_{\tilde{N}}) \le \rg_n(\tilde{N}), \] 
 where we denote by $E_{\tilde{N}}$ an elementary $\Lambda$-module attached to $\tilde{N}$. Summarising, we have shown that 
 \[ \rg_n(E_{\tilde{N}}) \le m_0(X_\infty) \cdot (p^{2n} - p^n) + n \cdot k + C, \] 
 Since ${\rg_n(E_{\tilde{N}}) \ge m_0(\tilde{N}) \cdot (p^{2n} - p^n)}$, it follows from our choice of $n$ (in particular, cf. property (c)) that 
 \[ m_0(\tilde{X}_\infty) = m_0(\tilde{N}) \le m_0(X_\infty). \] 
 As ${\tilde{X}_\infty \in U}$ had been chosen arbitrarily, this proves assertion (a) of the theorem. 
 
 Now suppose that ${m_0(\tilde{X}_\infty) = m_0(X_\infty)}$ for some fixed ${\tilde{X}_\infty \in U}$, and write the characteristic power series of $\tilde{N}$ as 
 \[ f_{\tilde{N}} = p^{m_0(\tilde{N})} \cdot g_{\tilde{N}}, \] 
 with $p \nmid \tilde{g}_N$. It follows from the proof of \cite[Theorem~4.6]{local_max} that either $g_{\tilde{N}}$ is in Weierstrass normal form with respect to $T_l$, say, 
 \[ g_{\tilde{N}} = T_l^{\tilde{k}} + \tilde{h}_{\tilde{k} - 1} T_l^{\tilde{k} - 1} + \ldots + \tilde{h}_0 \] 
 for suitable ${\tilde{h}_0, \ldots, \tilde{h}_{\tilde{k}-1} \in \Z_p[[T_1, \ldots, T_{l-1}]]}$, in which case we may conclude (from the same proof) that ${\tilde{k} \le k}$, or 
 \[ \rg_n(E_{\tilde{N}}) \ge m_0(\tilde{N}) \cdot (p^{2n} - p^n) + p^n (p-1). \] 
 In view of our choice of $n$, the latter alternative is not possible. Since moreover 
 \[ l_0(\tilde{N}) \le \tilde{k} \] 
 by \cite[Lemma~4.1]{local_max}, assertion (b) from the theorem follows. 
\end{proof} 

In the case ${l = 1}$ of $\Z_p$-covers of $X$, we can actually prove a more precise statement. 
\begin{thm} \label{thm:local_max2} 
    Assume that $l=1$ and let ${X_\infty \in \mathcal{E}^1(X)}$. Then there exists an integer ${m \in \N}$ such that the following statements hold: \begin{compactenum}[(a)] 
      \item For each ${\tilde{X}_\infty \in U := \mathcal{E}^1(X_\infty,m)}$, we have 
      \[ \mu(\tilde{X}_\infty) \le \mu(X_\infty). \] 
      \item ${\lambda(\tilde{X}_\infty) = \lambda(X_\infty)}$ for each ${\tilde{X}_\infty \in U}$ which satisfies ${\mu(\tilde{X}_\infty) = \mu(X_\infty)}$. 
    \end{compactenum} 
\end{thm} 
\begin{proof} 
  For statement (a), we could use the previous theorem, since 
  \[ \mu(N) = \mu(\varprojlim_m J(X_m))\] 
  holds also in the case ${l = 1}$ (see also Corollary~\ref{cor:zusammenhangzuN}). However, it is not hard to give an argument which will prove both (a) and (b). 
  
  Fix a pseudo-isomorphism ${\varphi: N \longrightarrow E_N}$, where $E_N$ is an elementary $\Lambda$-module. Then ${N^\circ = \ker(\varphi)}$ is trivial, i.e. $\varphi$ is an injection. Indeed, $N^\circ$ is actually equal to the maximal finite $\Lambda$-submodule of $N$; therefore the claim follows from Lemma~\ref{lemma:no-pseudo-null}. We stress that when compared to the proof of Theorem~\ref{thm:local_max1}, we have interchanged the roles of the set of definition and the image set in the definition of $\varphi$. 
  
  Moreover, the quotient $N/\nu_{2n,n}(T)$ is finite for each sufficiently large $n$ since $N$ is a torsion $\Lambda$-module and the polynomials $\nu_{n+1,n}(T)$ are pairwise coprime. 
  
  Letting 
  \[ \rg_n(N) := v_p(|N/\nu_{2n,n}(T)|) \] 
  (and similarly for other torsion $\Lambda$-modules), we have 
  \begin{align*} \rg_n(N) &= \rg_n(E_N) + \rg_n(N^\circ) \\ 
  & = \rg_n(E_N) + v_p(|N^\circ|), \end{align*} 
  where the last equality holds for each sufficiently large $n$ (see \cite[proof of Theorem~3.10]{local_beh} for the first equation). Since $N^\circ$ is trivial, we actually may conclude that 
  \[ \rg_n(N) = \rg_n(E_N) \] 
  for each ${n \in \N}$. 
  
  Now let $n$ be large enough such that \begin{compactenum}[(a)] 
    \item $\rg_n(N)$ is finite, and 
    \item $p^n (p-1) > n \cdot \lambda(N)$. 
  \end{compactenum} 
  As in the proof of Theorem~\ref{thm:local_max1} we can use the stabilisation property from Theorem~\ref{thm:fukuda} in order to define a neighbourhood ${U = \mathcal{E}^1(X_\infty,m)}$ of $X_\infty$ such that 
  \[ \rg_n(\tilde{X}_\infty) = \rg_n(X_\infty) \] 
  for each ${\tilde{X}_\infty \in U}$. By the above, since the maximal pseudo-null submodule of $\tilde{N}$ is also trivial, we have 
  \begin{align} \rg_n(E_{\tilde{N}}) &= \rg_n(E_N) \nonumber \\ 
    & = \mu(N) \cdot (p^{2n} - p^n) + \lambda \cdot n \end{align} 
for each ${\tilde{X}_\infty \in U}$, where $E_{\tilde{N}}$ is an elementary $\Lambda$-module attached to $\tilde{N}$, and where the last equation holds in view of the property (b) above (see the proof of \cite[Theorem~3.10]{local_beh}). 

Since ${\rg_n(E_{\tilde{N}}) \ge \mu(\tilde{N}) \cdot (p^{2n} - p^n)}$, the choice of $n$ implies that ${\mu(\tilde{N}) \le \mu(N)}$. This proves (a). 

Now suppose that ${\mu(\tilde{N}) = \mu(N)}$. If ${\lambda(\tilde{N}) \ge p^n(p-1)}$, then 
\[ \rg_n(E_{\tilde{N}}) \ge \mu(\tilde{N}) \cdot (p^{2n} - p^n) + p^n(p-1) \] 
by the proof of \cite[Theorem~3.10]{local_beh}. Since this is not possible in view of our choice of $n$, it follows that ${\lambda(\tilde{N}) < p^n(p-1)}$. Then the proof of \cite[Theorem~3.10]{local_beh} implies that 
\[ \rg_n(E_{\tilde{N}}) = \mu(\tilde{N}) \cdot (p^{2n} - p^n) + n \cdot \lambda(\tilde{N}). \] 
Since ${\rg_n(E_{\tilde{N}}) = \rg_n(E_N)}$ by the above, we may conclude that ${\lambda(\tilde{N}) = \lambda(N)}$. 
\end{proof} 

\begin{cor} 
  The $m_0$-invariant is bounded globally on $\mathcal{E}^l(X)$, i.e. there exists a constant $C$ such that 
  \[ m_0(X_\infty) \le C \] 
  for each ${X_\infty \in \mathcal{E}^l(X)}$. 
  
  Moreover, if the $m_0$-invariant is constant on $\mathcal{E}^l(X)$, then the $l_0$-invariant is bounded globally on $\mathcal{E}^l(X)$. 
\end{cor} 
\begin{proof} 
  This follows by combining Theorem~\ref{thm:local_max1} with Lemma~\ref{lemma:top_kompakt}. 
\end{proof} 
Again, we prove a more precise statement in the one-dimensional case. 
\begin{cor} \label{cor:unbounded}
  Let ${l =1}$, and let ${X_\infty \in \mathcal{E}^l(X)}$. We fix a neighbourhood $U$ of $X_\infty$ as in Theorem~\ref{thm:local_max2}. Then there exists a potentially smaller neighbourhood ${U' \subseteq U}$ of $X_\infty$ with the following property: 
  
  Either ${\mu(\tilde{X}_\infty) = \mu(X_\infty)}$ for each ${\tilde{X}_\infty \in U'}$, or $\lambda$ is unbounded on $U'$. 
\end{cor} 
\begin{proof} 
  We will use the notation from the proof of Theorem~\ref{thm:local_max2}. Suppose that $n$ has been chosen such that ${p^n (p-1) > n \cdot \lambda(N)}$, and that ${\mu(\tilde{N}) < \mu(N)}$. For any ${\tilde{X}_\infty \in U}$ with $\mu(\tilde{N})<\mu(N)$, we have that either ${\lambda(\tilde{N}) \ge p^n(p-1)}$, or the equation 
  \[ \rg_n(E_{\tilde{N}}) = \rg_n(E_N) \] 
  implies that 
  \begin{align*} \lambda(\tilde{N}) & \ge \lambda(N) + \Big\lfloor \frac{p^{2n}-p^n}{n} \Big\rfloor \cdot (\mu(N) - \mu(\tilde{N})) \\ 
  & \ge \Big\lfloor \frac{p^{2n} - p^n}{n} \Big\rfloor. \end{align*} 
  Here $\lfloor a \rfloor$ means the largest integer which is smaller than or equal to $a$, respectively. 
  
  Now choose an integer $n' > n$. Then we obtain a possibly smaller neighbourhood of $X_\infty$ such that 
  \[ \rg_{n'}(E_{\tilde{N}}) = \rg_{n'}(E_N) \] 
  for each $\tilde{X}_\infty$ which is contained in this smaller neighbourhood. In this neighbourhood, we will have that 
  \[ \lambda(\tilde{N}) \ge \min(p^{n'}(p-1), \Big\lfloor \frac{p^{2n'} - p^{n'}}{n'} \Big\rfloor). \] 
  Letting $n'$ tend to infinity, we may conclude that the $\lambda$-invariant is unbounded on a neighbourhood of $K_\infty$. 
\end{proof} 

\section{A weak control theorem}
Before we can prove a control theorem we need to introduce one further group associated to our voltage graphs: We write ${\textup{Pic}(X_m) = \textup{Div}(X_m)/\textup{Pr}(X_m)}$. Note that this is an infinite abelian group of $\Z_p$-rank $1$. We define 
\[ \textup{Pic}_\infty=\textup{Pic}(X_\infty)\otimes \Lambda. \] 
Analogously to Lemma \ref{lemma:projective-limit} one can show that ${\textup{Pic}_\infty\cong \varprojlim_m \textup{Pic}(X_m)}$. In particular, we have a exact sequence of $\Lambda$-modules
\[0\longrightarrow \textup{Pr}_\Lambda\longrightarrow \textup{Div}_\Lambda \longrightarrow \textup{Pic}_\infty\longrightarrow 0. \]

Recall that we have seen in the proof of Lemma \ref{lemma:projective-limit} that the natural map ${\textup{Div}^0_\Lambda\to \textup{Div}_\Lambda}$ is injective. Thus, we obtain a natural injection 
\[J_\infty\longrightarrow \textup{Pic}_\infty.\]
\begin{lemma}
\label{lemma:control-finitelevel}
The natural maps
\[r_m\colon \textup{Pic}_\infty/I_m\textup{Pic}_\infty\longrightarrow \textup{Pic}(X_m)\]
are isomorphisms.
\end{lemma}
\begin{proof}
 Consider the tautological commutative diagram
  \[\begin{tikzcd}
  0\arrow[r]&\textup{Pr}_\Lambda\arrow[r]\arrow[d,"h_m"]&\textup{Div}_\Lambda\arrow[r]\arrow[d,"g_m"]&\textup{Pic}_\infty\arrow[r]\arrow[d,"r_m"]&0\\
  0\arrow[r]&\textup{Pr}(X_m)\arrow[r]&\textup{Div}(X_m)\arrow[r]&\textup{Pic}(X_m)\arrow[r]&0
  \end{tikzcd}\]  
  It follows from the proof of Theorem~\ref{thm:jacobian} that $h_m$ and $g_m$ are surjective. Hence, $r_m$ is surjective. Applying the snake lemma we obtain a short exact sequence
  \[0\longrightarrow \ker(h_m)\longrightarrow \ker(g_m)\longrightarrow \ker(r_m)\longrightarrow 0. \]
  Recall that $\textup{Div}_\Lambda$ is $\Lambda$-free in the generators $v_{i,0}$. Likewise $\textup{Div}(X_m)$ is $\Lambda/I_m$-free in the same generators. It follows that $\ker(g_m)=I_m\textup{Div}_\Lambda$ and the above sequence becomes
  \[0\longrightarrow \textup{Pr}_\Lambda\cap I_m\textup{Div}_\Lambda\longrightarrow I_m\textup{Div}_\Lambda\longrightarrow \ker(r_m)\longrightarrow 0. \]
  Thus, $\ker(r_m)$ is the image of $I_m\textup{Div}_\Lambda$ in $\textup{Pic}_\infty$, and therefore $\ker(r_m)\cong I_m\textup{Pic}_\infty$. 
\end{proof}

Recall from Section~\ref{section:J(X_m)} that 
\[ J_\infty := J(X_\infty) \otimes \Lambda = \varprojlim_m J(X_m), \] 
and that ${I_m \subseteq \Lambda}$ is an ideal for each ${m \in \N}$. 
\begin{lemma}
\label{control-finite-level-ii}
The natural maps
\[r'_m \colon J_\infty/I_m\longrightarrow J(X_m)\]
are surjective. The $p$-rank of their kernels is bounded by $l$. 
\end{lemma}
\begin{proof}
Consider the commutative diagram
  \[\begin{tikzcd}
  0\arrow[r]&J_\infty\arrow[r]\arrow[d,"r'_m"]&\textup{Pic}_\infty\arrow[r]\arrow[d,"r_m"]&\Lambda/J_0\arrow[r]\arrow[d]&0\\
  0\arrow[r]&J(X_m)\arrow[r]&\textup{Pic}(X_m)\arrow[r]&\Lambda/J_0\arrow[r]&0
  \end{tikzcd}\] 
  The right vertical map is an isomorphism. Thus, $r'_m$ is surjective and \[ \ker(r'_m)=\ker(r_m)=(I_m\textup{Div}_\Lambda+\textup{Pr}_\Lambda)/\textup{Pr}_\Lambda.\] 
  Let $x$ be a generator of $\textup{Pic}_\infty/J_\infty$. Then $\ker(r'_m)/I_mJ_\infty$ is generated by \[ \{\omega_m(T_1)x,\dots ,\omega_m(T_l)x\}.\] 
\end{proof}
We are not only interested in the projection to finite level, but also to the projection to $\Z_p$-subcovers. Let ${Y_\infty}$ be a $\Z_p$-subcover of a given $\Z_p^l$-cover $X_\infty$. Let $\Lambda$ and $\Lambda(Y_\infty)$ be the corresponding Iwasawa algebras. Without loss of generality we can assume that the kernel of the induced map $\Lambda \longrightarrow \Lambda(Y_\infty)$ on the Iwasawa algebras is $(T_2, \dots , T_l)$. Let $J'_m\subseteq \Lambda$ be the ideal generated by $\omega_m(T_2),\ldots, \omega_m(T_l)$. We write $J_\infty(Y_\infty)$ and $J_\infty(X_\infty)$ e.t.c to make the corresponding graphs clear.
\begin{lemma} 
The natural map
\[t\colon \textup{Pic}_{\infty}(X_\infty)/J'_0\longrightarrow \textup{Pic}_\infty(Y_\infty)\]
is an isomorphism.
\end{lemma}
\begin{proof}
Consider the commutative diagram 
\[\begin{tikzcd}
  0\arrow[r]&\textup{Pr}_\Lambda(X_\infty)\arrow[r]\arrow[d,"h"]&\textup{Div}_\Lambda(X_\infty)\arrow[r]\arrow[d,"g"]&\textup{Pic}_\infty(X_\infty)\arrow[r]\arrow[d,"t"]&0\\
  0\arrow[r]&\textup{Pr}_\Lambda(Y_\infty)\arrow[r]&\textup{Div}_\Lambda(Y_\infty)\arrow[r]&\textup{Pic}_\infty(Y_\infty)\arrow[r]&0
  \end{tikzcd}\]
  Again all three vertical maps are surjective and the kernel of $g$ is $J'_0\textup{Div}_\Lambda(X_\infty)$. The rest of the proof is the same as for Lemma \ref{lemma:control-finitelevel}.
\end{proof}
Using the same ideas as in the proof of Lemma \ref{control-finite-level-ii} we obtain
\begin{lemma} \label{lemma:control} 
The natural map 
\[t'\colon J_\infty(X_\infty)/J'_0\longrightarrow J_\infty(Y_\infty)\]
is surjective and the kernel has $p$-rank at most $l-1$. 
\end{lemma}

\section{Examples} 
In this section we want to compute the (generalised) Iwasawa invariants of certain concrete examples. The starting point is the main conjecture (see Theorem~\ref{thm:main_conjecture}). In view of this result, we can compute the characteristic power series of a $\Z_p^l$-cover (if ${l = 1}$, then we have to keep in mind the Remark~\ref{rem:main_conjecture}). 

In this section, we consider a base graph of the following form: 
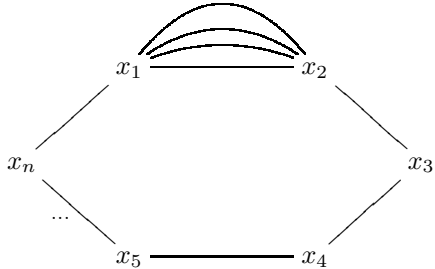
\vspace{1cm} 
\[ \xymatrix{&&&&\\ & x_1 \ar@{-}[rr] \ar@{-}@/^1.2pc/[rr] \ar@{-}@/^0.7pc/[rr] \ar@{-}@/^2pc/[rr] & & x_2 \ar@{-}[rd] & \\ 
x_n \ar@{-}[ur] & & & & x_3 \ar@{-}[ld] \\ 
& x_5 \ar@{-}[rr] \ar@{-}[ul]^{\ldots} & & x_4 &} \]
In other words, we let ${X = X_0}$ be a graph with $n$ edges ${x_1, \ldots, x_n}$, which are connected in a cycle, with one multiple edge, say, between $x_1$ and $x_2$. Let 
\[ \alpha \colon S \longrightarrow \Z_p^l \] 
be a voltage assignment. We write the Galois group multiplicatively. 
We have to study the $n \times n$-matrix 
\[ \Delta_\infty = D - A_\alpha = \begin{pmatrix} 
     a & x & 0 & 0 & 0 & -1 \\ 
     y & a & -1 & 0 & 0 & 0 \\ 
     0 & -1 & 2 & -1& 0 & 0 \\ 
     0 & 0 & \ddots & \ddots & \ddots & 0 \\ 
     0 & 0 & 0 & -1& 2 & -1 \\ 
     0 & 0 & 0 & 0 & -1 & 2 
\end{pmatrix}.\] 
Here we denote by 1 the trivial element of the group $\Z_p^l$, and $x,y \in \Z_p[\Z_p^l]$ are the group ring elements corresponding to the voltage assignment of the multi-edge. More precisely, if ${e_1, \ldots, e_k}$ denote the edges between $x_1$ and $x_2$, then 
\[ x = - \alpha(e_1) - \ldots - \alpha(e_k), \quad y = - \alpha(e_1)^{-1} - \ldots - \alpha(e_k)^{-1}, \] 
and ${a = k+1}$. 

We want to compute the determinant of $\Delta_\infty$, which is almost tridiagonal. To this purpose, we use an idea from the proof of \cite[Theorem~1]{wei}: let $\rho$ be the permutation matrix which has entries  
\[ \rho_{i,j} = \begin{cases} 1 & i = n, j = 1, \\ 
1 & j = i+1, \\ 
0 & \text{otherwise}. \end{cases} \] 
Then the determinant of $\rho$ is equal to $(-1)^{n-1}$, and the matrix product ${\Delta_\infty \cdot \rho}$ is of the form 
\[ \Delta_\infty \cdot \rho = \left(\begin{array}{cc|cccccc} 
   -1 & a & x & 0 & \ldots & \ldots & \ldots & 0 \\ 
   0 & y & a & -1 & 0 & \ldots & \ldots & 0 \\ \hline 
   0 & 0 & -1 & 2 & -1 & 0 & \cdots & 0 \\ 
   \vdots & \vdots & 0 & -1 & 2 & -1 & & \vdots \\ 
   \vdots & & \vdots & & \ddots & \ddots & \ddots & 0\\ 
   0 & \vdots & & & & -1 & 2 & -1 \\ 
   -1 & 0 & \vdots & & & & -1 & 2 \\ 
   2 & -1 & 0 & \cdots & & \cdots & 0 & -1 
\end{array}\right), \] 
which can be viewed as a block matrix $\begin{pmatrix} A & B \\ C & D \end{pmatrix}$. 

Now we apply \emph{Schur's determinant formula} (see \cite[Theorem~4.1]{Ouellette}): 
\[ \det\left(\begin{pmatrix} A & B \\ C & D \end{pmatrix} \right) = \det(D) \cdot \det(A - B \cdot D^{-1} \cdot C). \] 
It is easy to see that 
\[ D^{-1} = \begin{pmatrix} -1 & -2 & -3 & -4& \cdots & -n \\ 0 & -1 & \cdots & & & -(n-1) \\ 
\vdots & & & & & -(n-2) \\ 
& & & & & \vdots \\ 
0 & \cdots & & & & -1 \end{pmatrix} = (m_{ij}), \] 
where 
\[ m_{ij} = \begin{cases} 0 & i > j, \\ 
- j + i -1 & i \le j. \end{cases} \] 
It follows from a straight-forward computation that 
\[ A - BD^{-1} C = \begin{pmatrix} -1 + (n+1) x & a - nx \\ -n + (n+1)a & y - 1 - (a-1)n \end{pmatrix}. \] 
Computing the determinants, we may conclude that 
\[ \det(\Delta_\infty) = -[((n+1)y -1) \cdot x - y + (-a^2+2a-1)\cdot n + (1 - a^2)]. \] 
Here we have used that $\det(\rho) \cdot \det(D) = (-1)^{n-1} \cdot (-1)^{n-2} = -1$. 

In our examples, we first concentrate on the case ${l = 1}$. 
\begin{example} 
  The easiest case is $a = 2$ (i.e., no multiple edge). In this case ${y = x^{-1}}$, and the above formula simplifies to 
  \[ \det(\Delta_\infty) = 2 +x +x^{-1}. \] 
  In particular, the determinant does not depend on $n$ in this special case. Letting ${\tau = T+1}$ be a topological generator of the Galois group of $X_\infty$ over $X$, we have ${x = - \tau = -(T+1)}$ and we obtain the characteristic power series 
  \[ (T+1)^{-1} \cdot (2(T+1) - (T+1)^2 - 1) = (T+1)^{-1} \cdot (-T^2). \] 
  Note that $T+1$ is a unit in $\Lambda$. It follows from Remark~\ref{rem:main_conjecture} that the characteristic power series of $J_\infty$ is equal to $T$. 
  Thus, in this case we obtain ${\lambda(J_\infty)=1}$ independently of ${n = |X_0|}$. 
\end{example} 

\begin{example} \label{ex:mu} 
  In our second example we consider the case ${a = 3}$, i.e. we have two edges between $x_1$ and $x_2$. If the voltage assignments of these two edges are $\tau$ and $1$, respectively, then a similar computation yields that the characteristic power series is associated with 
  \[ (n+2) \cdot T^2. \] 
  Therefore ${\lambda(J_\infty) = 1}$ for all primes $p$, and ${\mu(J_\infty) = 0}$ except for the prime divisors of ${n+2}$. In particular, we see that this family of examples includes voltage covers with \emph{arbitrarily large $\mu$-invariant} (for fixed prime $p$, let ${n = p^r - 2}$ for some large integer $r$). 
  
  If the voltage assignments are $\tau$ and $\tau$, then we obtain that the characteristic power series of $J_\infty$ is associated with 
  \[ 2 T^2, \] 
  i.e. it does not depend on the number $n$ of vertices of $X_0$. 
  
  Finally, suppose that $\alpha(e_1) = \tau$ and $\alpha(e_2) = \tau^2$. Then the characteristic power series is associated with 
  \[ T^2 \cdot (T^2 + (6+n)T +6+n). \] 
  Therefore ${\mu(J_\infty) = 0}$ for all primes $p$, and 
  \[ \lambda(J_\infty) = \begin{cases} 3 & p\mid (n+6); \\ 
  1 & p \nmid (n+6). \end{cases} \] 
\end{example} 

\begin{example} 
  Now suppose that $a$ is arbitrary, and that each $\alpha(e_i)$ is equal to $\tau$, ${i = 1, \ldots, a-1}$. Then we obtain that 
  \[ \det(\Delta_\infty) = -(T+1)^{-1}(a-1)T^2, \] 
  i.e. the characteristic polynomial of $J_\infty$ is associated with ${(a-1)T}$ in this case; in particular this polynomial does not depend on the number of vertices of $X_0$. 
\end{example} 

\begin{example} 
  Now suppose that $a$ is arbitrary and that there exists some integer ${b \le a-1}$ such that 
  \[ \alpha(e_1) = \ldots = \alpha(e_b) = \tau, \quad \alpha(e_{b+1}) = \ldots = \alpha(e_{a-1}) = 1. \] 
  Then we obtain that $\det(\Delta_\infty)$ is associated to 
  \[ T^2((ab - b^2 - b) \cdot n + ab - b^2). \] 
  In particular, this determinant is independent of the number $n$ of vertices of $X_0$ if and only if ${b = 0}$ or ${a = b+1}$. This fact is in accordance with our previous examples. Note that we did not consider examples with ${b = 0}$ since in that case the hypothesis from Lemma~\ref{lemma:connected} is not satisfied, which ensures that all the $X_n$ will be connected. 
\end{example} 

\begin{example} \label{ex:l=2} 
  In our last example we consider a $\Z_p^2$-cover of $X_0$. Let $\sigma$ and $\tau$ denote topological generators of ${\Gal(X_\infty/X) \cong \Z_p^2}$. Consider the case 
  \[ a = 3, \quad \alpha(e_1) = \tau, \quad \alpha(e_2) = \sigma. \] 
  Write ${\tau = T+1}$ and ${\sigma = S+1}$. Then we obtain that the characteristic power series of $J_\infty$ (which is associated with $\det(\Delta_\infty)$ in this example, since ${l > 1}$, i.e. Remark~\ref{rem:main_conjecture} does not apply) gives 
  \[ - \frac{1}{(S+1)(T+1)} \cdot \left((n+2)(T-S)^2 + 2ST + ST^2 + S^2 T\right). \] 
  We see that ${m_0(J_\infty) = 0}$. We will show in the next section that $l_0(J_\infty)$, however, is positive. 
\end{example}

\section{A non-trivial $l_0$-invariant} 
The main goal of this section is to provide an example of a $\Z_p^l$-cover $X_\infty/X$, ${l > 1}$, such that ${l_0(J(X_\infty)) > 0}$. 
For any $\Z_p^l$-cover $X_\infty/X$, we denote by $\mathcal{E}^{\subseteq X_\infty}(X)$ the set of $\Z_p$-covers of $X$ which are contained in $X_\infty$. We make use of the following 
\begin{thm} \label{thm:l0} 
  Let $X_\infty/X$ be a $\Z_p^2$-cover, and suppose that $X_\infty$ contains two $\Z_p$-subcovers $Y_1$ and $Y_2$ of $X$ such that 
  \[ \mu(Y_1) = 0, \quad \mu(Y_2) > 0. \] 
  Then $\lambda$ is unbounded on $\mathcal{E}^{\subseteq X_\infty}(X)$, and ${l_0(J_\infty) > 0}$. 
\end{thm} 
\begin{example} 
  Let $X_\infty/X$ be the $\Z_p^2$-cover from Example~\ref{ex:l=2}, with ${n = 3}$ and ${a = 3}$. Then $X_\infty$ contains the two $\Z_p$-voltage covers $Y_1$ and $Y_2$ which are given as follows: for $Y_1$, we have 
  \[ \alpha(e_1) = \tau, \quad \alpha(e_2) = 1 \] 
  (this corresponds to the choice ${\Gal(X_\infty/Y_1) = \langle \sigma \rangle}$), and for $Y_2$ we have 
  \[ \alpha(e_1) = \alpha(e_2) = \tau \] 
  (this corresponds to ${\Gal(X_\infty/Y_2) = \langle \sigma \tau^{-1}\rangle}$). 
  
  Then we know from Example~\ref{ex:mu} that the characteristic power series of $J(Y_1)$ is associated with $5T$, and the characteristic power series of $J(Y_2)$ is associated with $2T$. In particular, if we consider ${p = 5}$, then ${\mu(Y_1) > 0}$ and ${\mu(Y_2) = 0}$, i.e. the hypotheses from Theorem~\ref{thm:l0} are satisfied. This can be seen also by looking at the characteristic power series ${F = F(S,T)}$ of $J_\infty$ (see Example~\ref{ex:l=2}): modulo ${S = \sigma - 1}$ we obtain 
  \[ (n+2) T^2, \] 
  which is clearly divisible by $p = 5$ since ${n = 3}$ by assumption. On the other hand, if we let ${S = T}$ in $F$, then we obtain a polynomial which is associated with 
  \[ 2 T^2 + 2T^3 = T^2 (T+2) \sim T^2 \] 
  for ${p \ne 2}$; therefore $\mu(Y_2) = 0$. One can even see that ${l_0(F) > 0}$ directly here, since ${F \in (p, S)}$ for the choices ${n = 3}$ and ${p = 5}$. 
\end{example} 

The following lemma will be the crucial ingredient in the proof of Theorem~\ref{thm:l0}. 
\begin{lemma} \label{lemma:l0} 
  Let $X_\infty/X$ be a $\Z_p^2$-cover. 
  
  If ${\mu(Y) = 0}$ for any $\Z_p$-cover ${Y \in \mathcal{E}^{\subseteq X_\infty}(X)}$ of $X$, then ${\mu(\tilde{Y}) = 0}$ for all but finitely many ${\tilde{Y} \in \mathcal{E}^{\subseteq X_\infty}(X)}$. 
\end{lemma} 
\begin{proof} 
  Fix ${Y \in \mathcal{E}^{\subseteq X_\infty}(X)}$ such that ${\mu(Y) = 0}$. In what follows we will write $\Lambda_1$ for any one-dimensional Iwasawa algebra, ${\Lambda = \Z_p\llbracket S, T \rrbracket }$, and we do not abbreviate ${J_\infty(X_\infty) := J(X_\infty) \otimes \Lambda}$ to $J_\infty$ for clarity. 
  Moreover, for any ${Y \in \mathcal{E}^{\subseteq X_\infty}(X)}$, we let $J_\infty(Y)$ be the $\Lambda_1$-module ${J(Y) \otimes \Lambda_1}$, where ${\Lambda_1 = \Z_p\llbracket \Gal(Y/X)\rrbracket }$. 
  
  Suppose that the topological generators $\sigma$ and $\tau$ of ${\Gal(X_\infty/X) \cong \Z_p^2}$ have been chosen such that ${Y = X_\infty^{\langle \sigma \rangle}}$, and let ${S = \sigma + 1}$. Then Lemma~\ref{lemma:control} implies that the kernel of the canonical map 
  \[ t' \colon J_\infty(X_\infty)/S \longrightarrow J_\infty(Y) \] 
  has $p$-rank at most 1. Therefore the finitely generated ${\Lambda_1}$-module $J_\infty(X_\infty)/S$ is ${\Lambda_1}$-torsion and has $\mu$-invariant 0. This means that we can choose an annihilator ${f \in {\Lambda_1}}$ of $J_\infty(X_\infty)/S$ which is not divisible by $p$. 
  
  Now consider the $\Lambda $-module $J_\infty(X_\infty)$, and suppose that ${m_0(X_\infty) > 0}$. Since the Fitting ideal $\textup{Fitt}_{\Lambda}(J_\infty(X_\infty))$ is contained in the annihilator ideal of $J_\infty(X_\infty)$, this assumption would imply that 
  \[ \textup{Fitt}_{\Lambda }(J_\infty(X_\infty)) \subseteq (p). \] 
  But it follows from the general properties of Fitting ideals (see also \cite[Proposition~2.1]{kleine-matar}) that \begin{compactenum}[(i)] 
   \item $\pi(\textup{Fitt}_{\Lambda}(J_\infty(X_\infty)) = \textup{Fitt}_{\Lambda_1}(J_\infty(X_\infty)/S)$ and 
   \item $\textup{Ann}_{\Lambda_1}(J_\infty(X_\infty)/S)^l \subseteq \textup{Fitt}_{\Lambda_1}(J_\infty(X_\infty)/S)$ for all sufficiently large ${l \in \N}$, where 
   \[ \pi \colon \Lambda \longrightarrow {\Lambda_1} \] 
   denotes the surjection induced by mapping $S$ to 0, and where 
\[ \textup{Ann}_{\Lambda_1}(J_\infty(X_\infty/S))\] 
denotes the annihilator ideal. 
  \end{compactenum} 
  Therefore each annihilator of $J_\infty(X_\infty)/S$ would be divisible by $p$, which is a contradiction to the existence of the annihilator $f$ above. 
  
  We have shown that ${m_0(X_\infty) = 0}$. Now let ${\tilde{Y} \in \mathcal{E}^{\subseteq X_\infty}(X)}$ be arbitrary, and choose topological generators $\tilde{\sigma}$ and $\tau$ of ${\Gal(X_\infty/X)}$ such that ${\tilde{Y} = X_\infty^{\langle \tilde{\sigma} \rangle}}$. Since ${m_0(X_\infty) = 0}$, we can choose an annihilator ${F \in \Lambda}$ which is not divisible by $p$. In fact, $F$ can be chosen such that $F$ is not contained in the ideal $(p, \tilde{S})$, ${\tilde{Y} = X_\infty^{\langle \tilde{S} + 1 \rangle}}$, for all but finitely many choices of $\tilde{Y}$. Indeed, the characteristic power series $F_{J_\infty(X_\infty)}$ of $J_\infty(X_\infty)$ is contained in finitely many of the ideals $(p, \tilde{S})$ (note that each such ideal contributes to the $l_0$-invariant of $X_\infty$). Suppose now that $\tilde{Y}$ has been chosen such that $F_{J_\infty(X_\infty)}$ is not contained in the corresponding ideal $(p, \tilde{S})$. Since $J_\infty(X_\infty)$ does not contain any non-trivial pseudo-null submodules by Lemma~\ref{lemma:no-pseudo-null}, it follows from the general structure of $\Lambda$-modules that $F_{J_\infty(X_\infty)}$ annihilates $J_\infty(X_\infty)$ (see also \cite[Proposition~2.1(3)]{kleine-matar}). Therefore we can choose ${F = F_{J_\infty(X_\infty)}}$ above; this proves our claim. 
  
  Again, a suitable power of $F$ (the exponent of this power depends on the number of generators of $J_\infty(X_\infty)$ over $\Lambda$) is contained in the Fitting ideal $\textup{Fitt}_{\Lambda}(J_\infty(X_\infty))$, and 
  \[ \tilde{\pi}(F) \in \textup{Fitt}_{\Lambda_1}(J_\infty(X_\infty)/\tilde{S}) \subseteq \textup{Ann}_{\Lambda_1}(J_\infty(X_\infty)/\tilde{S}), \] 
  where now $\tilde{\pi} \colon \Lambda \longrightarrow {\Lambda_1}$ is induced by mapping $\tilde{S}$ to zero. 
  
  If $\tilde{\pi}(F)$ is not divisible by $p$ for our particular choice of $\tilde{S}$ (by the above, this can be achieved for all but finitely many ${\tilde{Y} \in \mathcal{E}^{\subseteq X_\infty}(X)}$), then it follows that ${\mu(J_\infty(X_\infty)/\tilde{S}) = 0}$, i.e.  $J_\infty(X_\infty)/\tilde{S}$ is a finitely generated $\Z_p$-module. By Lemma~\ref{lemma:control}, the canonical map 
  \[ J_\infty(X_\infty)/\tilde{S} \longrightarrow J_\infty(\tilde{Y}) \] 
  is surjective. This shows that $J_\infty(\tilde{Y})$ is also finitely generated over $\Z_p$, and therefore ${\mu(J_\infty(\tilde{Y})) = 0}$. 
\end{proof} 

\begin{proof}[Proof of Theorem~\ref{thm:l0}] 
  Since ${\mu(Y_1) = 0}$, it follows from Lemma~\ref{lemma:l0} that ${\mu(\tilde{Y}) = 0}$ for all but finitely many ${\tilde{Y} \in \mathcal{E}^{\subseteq X_\infty}(X)}$. Now fix $Y_2$ such that ${\mu(Y_2) > 0}$, and let ${U' = \mathcal{E}(Y_2, n)}$ be a neighbourhood as in Corollary~\ref{cor:unbounded}. Then ${\mu(\tilde{Y}) = 0}$ for all but finitely many ${\tilde{Y} \in U'}$. The first assertion of Theorem~\ref{thm:l0} now follows from Corollary~\ref{cor:unbounded}. 
  
  Moreover, it follows from the proof of Lemma~\ref{lemma:l0} that ${m_0(X_\infty) = 0}$, and that ${\mu(Y_2) > 0}$ is possible only if the characteristic power series ${F_{J_\infty(X_\infty)} \in \Lambda}$ is contained in the ideal $(p, \tilde{S})$, where ${\tilde{S} = \tilde{\sigma} - 1}$ for some topological generator $\tilde{\sigma}$ of $\Gal(X_\infty/Y_2)$. This implies that ${l_0(J_\infty(X_\infty)) > 0}$, by the definitions. 
\end{proof} 
\bibliography{references} 

\begin{thebibliography}{10}

\bibitem{babaichev}
Vladimir~A. {Baba\u{\i}cev}.
\newblock {On some questions in the theory of $\Gamma$-extensions of algebraic
  number fields. II.}
\newblock {\em {Math. USSR, Izv.}}, 16:675--685, 1976.

\bibitem{corry-Perkinson}
Scott Corry and David Perkinson.
\newblock {\em Divisors and sandpiles}.
\newblock American Mathematical Society, Providence, RI, 2018.
\newblock An introduction to chip-firing.

\bibitem{cuoco-monsky}
Albert~A. Cuoco and Paul Monsky.
\newblock Class numbers in {${\bf Z}^{d}_{p}$}-extensions.
\newblock {\em Math. Ann.}, 255(2):235--258, 1981.

\bibitem{dubose-vallieres}
Sage {Dubose} and Daniel {Vallières}.
\newblock On $\mathbb{Z}_l^d$-towers of graphs.
\newblock {\em Preprint}, 2022.

\bibitem{fukuda}
Takashi Fukuda.
\newblock Remarks on {$\bold Z_p$}-extensions of number fields.
\newblock {\em Proc. Japan Acad. Ser. A Math. Sci.}, 70(8):264--266, 1994.

\bibitem{diss-gonet}
Sophia~Rose Gonet.
\newblock {\em Jacobians of {F}inite and {I}nfinite {V}oltage {C}overs of
  {G}raphs}.
\newblock ProQuest LLC, Ann Arbor, MI, 2021.
\newblock Thesis (Ph.D.)--The University of Vermont and State Agricultural
  College.

\bibitem{green_73}
Ralph {Greenberg}.
\newblock {The Iwasawa invariants of $\Gamma$-extensions of a fixed number
  field.}
\newblock {\em {Am. J. Math.}}, 95:204--214, 1973.

\bibitem{iwasawa}
Kenkichi {Iwasawa}.
\newblock On {$\Gamma $}-extensions of algebraic number fields.
\newblock {\em Bull. Amer. Math. Soc.}, 65:183--226, 1959.

\bibitem{iwasawa_mu}
Kenkichi Iwasawa.
\newblock On the {$\mu $}-invariants of {$Z_{\ell}$}-extensions.
\newblock In {\em Number theory, algebraic geometry and commutative algebra, in
  honor of {Y}asuo {A}kizuki}, pages 1--11. Kinokuniya, Tokyo, 1973.

\bibitem{local_beh}
S\"oren {Kleine}.
\newblock {Local behavior of Iwasawa's invariants.}
\newblock {\em {Int. J. Number Theory}}, 13(4):1013--1036, 2017.

\bibitem{local_max}
S\"{o}ren Kleine.
\newblock Generalised {I}wasawa invariants and the growth of class numbers.
\newblock {\em Forum Math.}, 33(1):109--127, 2021.

\bibitem{kleine-matar}
Sören {Kleine} and Ahmed {Matar}.
\newblock {Boundedness of Iwasawa invariants of fine Selmer groups and Selmer
  groups}.
\newblock {\em preprint}, 2022.

\bibitem{vallieres-lei}
Antonio Lei and Daniel {Vallières}.
\newblock The non-$\ell$-part of the number of spanning trees in abelian
  $\ell$-towers of multigraphs.
\newblock {\em Preprint}, 2022.

\bibitem{vallieres2}
Kevin McGown and Daniel Valli\`eres.
\newblock {On abelian {$\ell$}-towers of multigraphs II}.
\newblock {\em Ann. Math. Qu\'{e}.}

\bibitem{vallieres3}
Kevin McGown and Daniel Valli\`eres.
\newblock {On abelian {$\ell$}-towers of multigraphs III}.
\newblock {\em Ann. Math. Qu\'{e}.}

\bibitem{nsw}
J\"urgen {Neukirch}, Alexander {Schmidt}, and Kay {Wingberg}.
\newblock {\em {Cohomology of number fields. 2nd ed.}}
\newblock Berlin: Springer, 2nd edition, 2008.

\bibitem{Ouellette}
Diane~Val\'{e}rie Ouellette.
\newblock Schur complements and statistics.
\newblock {\em Linear Algebra Appl.}, 36:187--295, 1981.

\bibitem{vallieres1}
Daniel Valli\`eres.
\newblock On abelian {$\ell$}-towers of multigraphs.
\newblock {\em Ann. Math. Qu\'{e}.}, 45(2):433--452, 2021.

\bibitem{wash}
Lawrence~C. {Washington}.
\newblock {\em {Introduction to cyclotomic fields. 2nd ed.}}
\newblock New York, NY: Springer, 1997.

\bibitem{wei}
Yunlan Wei, Xiaoyu Jiang, Zhaolin Jiang, and Sugoog Shon.
\newblock Determinants and inverses of perturbed periodic tridiagonal
  {T}oeplitz matrices.
\newblock {\em Adv. Difference Equ.}, pages Paper No. 410, 11, 2019.

\end{thebibliography}
\bibliographystyle{plain} 

\end{document}